\documentclass{article}
\usepackage[english]{babel}
\usepackage[letterpaper,top=2cm,bottom=2cm,left=3cm,right=3cm,marginparwidth=1.75cm]{geometry}

\usepackage{amssymb}
\usepackage{xcolor}
\usepackage{subcaption}
\usepackage{empheq}
\usepackage[colorlinks=true, allcolors=blue]{hyperref}
\usepackage{algorithm}
\usepackage{algpseudocode}
\usepackage{enumitem}
\usepackage{amsthm}

\usepackage{bm,array,tabularx,multirow}
\usepackage{enumitem}

\newcolumntype{L}[1]{>{\raggedright\arraybackslash}p{#1}}
\newcolumntype{Y}{>{\raggedright\arraybackslash}X}

\setlist[itemize]{leftmargin=1.2em, topsep=1pt, itemsep=1pt, parsep=0pt, partopsep=0pt}

\newcommand{\shft}{c}

\usepackage{amsmath}
\usepackage{url}

\newcommand{\cvx}{g}

\newcommand{\sout}{s}

\newcommand{\cin}{r}
\newcommand{\Cin}{R}

\newcommand{\sinn}{C}
\newcommand{\ssinn}{c}
\newcommand{\cout}{h}

\newcommand{\Pred}{\mathrm{Pred}}
\newcommand{\Act}{\mathrm{Act}}

\newcommand{\R}{\mathbb{R}}

\newtheorem{theorem}{Theorem}
\newtheorem{corollary}{Corollary}
\newtheorem{lemma}{Lemma}
\newtheorem{definition}{Definition}
\newtheorem{assumption}{Assumption}
\newtheorem{proposition}{Proposition}

\newtheorem{remark}{Remark}

\renewenvironment{proof}{\par\textcolor{gray}{\textit{\textbf{Proof.} }}}{\hfill$\square$\par}

\numberwithin{theorem}{section}
\numberwithin{corollary}{section}
\numberwithin{lemma}{section}
\numberwithin{definition}{section}
\numberwithin{assumption}{section}
\numberwithin{proposition}{section}
\numberwithin{cond}{section}
\numberwithin{remark}{section}
\numberwithin{example}{section}

\DeclareMathOperator{\dom}{dom}

\title{A Proximal Method for Composite Optimization\\with Smooth and Convex Components}
\date{}

\usepackage{authblk}

\author{Samet Uzun$^{1}$}
\author{Dayou Luo$^{2}$}
\author{Beh{\c{c}}et A{\c{c}}{\i}kme{\c{s}}e$^{1}$}
\author{Aleksandr Y. Aravkin$^{2}$}

\affil{
\vspace{-0.2cm}
\begin{center}
$^{1}$Department of Aeronautics \& Astronautics
\end{center}

\vspace{-0.4cm}
\begin{center}
$^{2}$Department of Applied Mathematics
\end{center}

\vspace{-0.4cm}
\begin{center}
University of Washington, Seattle, WA, USA
\end{center}

\vspace{-0.4cm}
\begin{center}
(email: \texttt{\{samet, dayoul, behcet, saravkin\}@uw.edu})
\end{center}
}

\begin{document}
\maketitle

\vspace{-0.6cm}

\begin{abstract}
We introduce prox-convex for minimizing $F(x)=\cvx(x)+\cout(\sinn(x))+\sout(\Cin(x))$, 
where $\cvx$ and $\cout$ are convex, $\sinn$ and $\sout$ are smooth, and each component of $\Cin$ is convex (possibly nonsmooth). 
Here $\cvx$ captures general convex objectives and indicator functions for convex constraints, while the composite template simultaneously models convex penalties on smooth features ($\cout\circ\sinn$) and smooth couplings of convex (possibly nonsmooth) features ($\sout\circ\Cin$).
Each prox-convex step forms a convex subproblem by linearizing only the smooth maps while preserving the existing convex structure. The resulting subproblem is made strongly convex with the proximal metric $Q_k=\mu_kI+H_k^+\succ0$ where $\mu_k$ is adapted using an implicit trust region strategy, and $H_k^+\succeq0$ is an optional curvature term for local acceleration. Under mild Lipschitz/smoothness and a per-coordinate monotone-or-smooth condition, we prove subdifferential regularity, derive two-sided quadratic model error bounds with explicit constants, and obtain sufficient decrease with $O(\varepsilon^{-2})$ complexity for driving the norm of the metric prox-gradient below $\varepsilon$. 
Furthermore, a local error-bound condition for $F$ guarantees a metric step-size error bound and hence local $Q$-linear convergence of the function values.
Using the Taylor-like model framework of Drusvyatskiy, Ioffe, and Lewis, we show that every cluster point of the iterates is limiting-stationary; under our regularity conditions, this further implies Fr\'echet stationarity. The same framework also establishes robustness to inexact subproblem solves and justifies a model-decrease termination rule.
%
\end{abstract}

\section{Introduction}
We consider composite problems
\begin{equation} \label{eq:def_F}
\min_{x\in\mathbb{R}^m}\; F(x)\;=\;\cvx(x)\;+\;\cout(\sinn(x))\;+\;\sout(\Cin(x)),
\end{equation}
where 
$\cvx:\mathbb{R}^m \to \mathbb{R}\cup\{+\infty\}$ is proper, closed, and convex; 
$\cout : \mathbb{R}^d \to \mathbb{R}$ is finite-valued, closed and convex; 
$\sinn : \mathbb{R}^m \to \mathbb{R}^d$ and $\sout : \mathbb{R}^n \to \mathbb{R}$ are $\mathcal{C}^1$-smooth;
$\Cin:\mathbb{R}^m \to \mathbb{R}^n$ with $\Cin(x)=(\cin_1(x),\ldots,\cin_n(x))$ and each $\cin_i : \mathbb{R}^m \to \mathbb{R}$ is finite-valued, closed and convex. 
For the convergence analysis, we further impose mild technical regularity and Lipschitz assumptions, which are collected in Section~\ref{sec:pcx}.
This class generalizes well-studied {\it convex-composite} problems because of the term $\sout(\Cin(x))$, which is an essential feature for motivating examples.  

\begin{remark}[Modeling breadth]
The decomposition \eqref{eq:def_F} covers a wide range of structured objectives.
Here, $\cvx$ collects convex objectives and indicator functions of convex constraints;
$\cout(\sinn(x))$ captures convex penalties applied to smooth (possibly nonlinear) feature maps;
and $\sout(\Cin(x))$ models smooth couplings of convex (possibly nonsmooth) features such as predicate functions used in temporal logic specifications.

Importantly, the split between $\Cin$ and $\sout$ is \emph{not unique} and can be chosen to preserve convex structure.
Only those features that we want to keep as convex coordinates need to appear inside $\Cin$.
Any additional smooth dependence on $x$ (even if not convex) can be absorbed into the smooth outer map by augmenting the inner map with the identity.
Concretely, if a modeling term has the form $\widetilde{s}(\bar R(x), q(x))$ where $\bar R$ is convex (possibly nonsmooth) and $q$ is smooth, then it can be written as
$\sout(\Cin(x))$ by setting $\Cin(x)=(\bar R(x),x)$ and defining $\sout(y,z)=\widetilde{s}(y,q(z))$.
Thus, one may assume without loss of modeling generality that $\Cin$ collects the convex building blocks, while all smooth components are handled by $\sout$.

Representative instances covered by \eqref{eq:def_F} are summarized in Table~\ref{tab:modeling_examples}.
\end{remark}

\begin{table}[!hbt]
\centering
\scriptsize
\setlength{\tabcolsep}{4pt}
\renewcommand{\arraystretch}{1.18}
\setlength{\emergencystretch}{2em}
\begin{tabularx}{\linewidth}{|L{0.10\linewidth}|L{0.266\linewidth}|Y|}
\hline 
\textbf{Component} & \textbf{Modeling role} & \textbf{Typical instances} \\ 
\hline 
\multirow{2}{*}{$\cvx(x)$} & \textbf{Convex objectives} & $\tfrac12\|Ax-b\|_2^2,\ \|Ax-b\|_2,\ a^\top x,\ \lambda\|x\|_1,\ \lambda\|x\|_2,$ TV, nuclear norm $\|X\|_\ast$ \\ 
\cline{2-3} 
& \textbf{Convex constraints} via indicator & $\delta_{\mathcal X}(x)$ (compact convex sets: boxes/balls/polytopes/simplices) \newline $\delta_{\mathcal K}(x)$ (SOC/SDP cones), indicator of affine equalities/dynamics \\ 
\hline 
\multirow{2}{*}{\cout(\sinn(x))} & \textbf{Composition of Convex loss with smooth features} & Nonlinear regression: $\tfrac12\|\sinn(x)\|_2^2,\ \|\sinn(x)\|_1,\ \|\sinn(x)\|_\infty$ \newline Robust fitting: Huber loss on $\sinn(x)$ \\ 
\cline{2-3} 
& \textbf{Exact penalization} of smooth (possibly nonconvex) constraints & For sufficiently large $w$: \newline Equality constraints $\big(\sinn(x) = 0\big)$: $w \| \sinn(x) \|_1$ \newline Inequality constraints $\big(\sinn(x) \!\le\! 0\big)$: $w \sum_i (\ssinn_i(x))_+$, where $(t)_+\!=\!\max\{0,t\}$ \\
\hline 
\multirow{4}{*}{\sout(\Cin(x))} & \textbf{Temporal Logic Specifications} with convex predicate functions & \textbf{Smooth and Exact Parameterizations of the Specifications \cite{uzun2024optimization}} \\ 
& \textbf{Disjunction (OR)} $\displaystyle \bigvee_{i=1}^n\big(\cin_i(x)\le 0\big)$
& \textbf{Disjunction (OR)} $-{}^{\vee}h_{p,w}^{\shft}\big(-\cin_1(x),\dots,-\cin_n(x)\big)\le 0 \iff \prod_{i=1}^n \big(\cin_i(x)\big)_+^2 = 0$ \\ 
& \textbf{Implication} $\displaystyle \big(\cin_1(x) < 0\big)\Rightarrow\big(\cin_2(x)\le 0\big)$
& \textbf{Implication} $-{}^{\vee}h_{p,w}^{\shft}\big(\cin_1(x),\,-\cin_2(x)\big)\le 0 \iff \big(-\cin_1(x)\big)_+^2\,\big(\cin_2(x)\big)_+^2 = 0$ \\ 
& \textbf{Compound}
$\displaystyle \big(\cin_1(x) < 0 \vee \cin_2(x) < 0\big)\Rightarrow$
$\big(\cin_3(x)\le 0 \wedge \cin_4(x)\le 0\big)$
& \textbf{Compound}
$
\begin{aligned}
&-{}^{\vee}h_{p,w}^{\shft}\Big( {}^{\wedge}h_{p,w}^{\shft}\big(\cin_1(x), \cin_2(x)\big),\, {}^{\wedge}h_{p,w}^{\shft}\big(-\cin_3(x), -\cin_4(x)\big)\Big)\le 0 \iff \\
&\Big(\big(-\cin_1(x)\big)_+^2 + \big(-\cin_2(x)\big)_+^2\Big) \Big(\big(\cin_3(x)\big)_+^2 + \big(\cin_4(x)\big)_+^2\Big) = 0
\end{aligned}
$ 
where $(t)_+ = \max\{0,t\}$ \\ 
\hline
\end{tabularx}
\caption{Modeling examples covered by the decomposition $F(x)=\cvx(x)+\cout(\sinn(x))+\sout(\Cin(x))$.}
\label{tab:modeling_examples}
\end{table}

For this broad class of problems, we propose a prox-convex algorithm that \emph{linearizes only the smooth maps} while \emph{preserving all convex structure}. At a current point $x_k$, we form a convex model $F(x;x_k)$ and solve the strongly convex subproblem:
\[
x_{k+1} = \arg\min_{x}\;\Big\{
F(x;x_k)\;+\;\tfrac12\|x-x_k\|_{Q_k}^2
\Big\},
\]
with the positive-definite proximal metric $Q_k:= \mu_kI+H_k^+\succ0$. 
Here $\mu_k>0$ is chosen by a predicted/actual reduction ratio test to enforce sufficient decrease, and $H_k^+\succeq0$ optionally injects curvature from $\mathcal C^2$ components of the smooth maps, with $H_k^+=0$ when such information is not available. 
This design keeps $\cvx$ and $\cout$ exact and retains the convex inner map $\Cin(\cdot)$ inside the linearized $\sout$ term, while the quadratic regularizer ensures stability and strong convexity of the subproblem.

Under standard boundedness and Lipschitz model-error conditions, accepted steps of prox-convex monotonically decrease $F$ and drive the metric prox-gradient
\[
\mathcal{G}_{Q_k}(x_k)\ :=\ Q_k\,(x_k-x_{k+1})
\]
to zero, and moreover $\min_{0\le j<N}\|\mathcal G_{Q_j}(x_j)\| = O(N^{-1/2})$ so that  $O(\varepsilon^{-2})$ accepted steps suffice to obtain $\|\mathcal{G}_{Q_k}(x_k)\|\le\varepsilon$.
A local error-bound condition further yields a metric step-size error bound and thus local $Q$-linear convergence of the function values.
The Taylor-like framework of Drusvyatskiy, Ioffe, and Lewis~\cite{drusvyatskiy2021nonsmooth}, transfers this residual control to stationarity: every cluster point is \emph{limiting} stationary and, by our subdifferential-regularity result, also \emph{Fr\'echet} stationary. The same framework also justifies robustness to inexact subproblem solves and a model-decrease stopping rule. Curvature terms $H_k^+$ provide local acceleration without weakening global guarantees. 

At a high level, the contributions of this paper are:
(i) a structure-preserving prox-convex subproblem that avoids unnecessary linearization of convex components,
(ii) an adaptive proximal metric with optional second-order information to improve conditioning and local progress, and
(iii) a convergence theory with explicit quantitative guarantees.
Additional details about these contributions follow the literature survey below. 

\paragraph{Related Works:} The development of algorithms for nonsmooth, nonconvex optimization has a rich history, evolving from classical methods for nonlinear programming to modern techniques designed for the large-scale, structured problems common in data science. Our proposed prox-convex algorithm builds upon several key lines of research, which we now survey to place our contribution in context.

The core idea of iteratively minimizing a simplified model of a difficult problem has deep roots in numerical optimization. 
For nonlinear least-squares $\min_x \tfrac12\|f(x)\|^2$, the \textbf{Gauss--Newton method} generalizes Newton's method \cite[Section 3.3]{nocedal2006numerical} by linearizing the inner function to create a tractable subproblem \cite[Section 10.3]{nocedal2006numerical}. 
It is a classical workhorse for nonlinear least-squares and nonlinear regression, widely used for parameter estimation in statistical models and inverse problems \cite{bjorck1996numerical}.
The \textbf{Levenberg--Marquardt method} \cite{levenberg1944method,marquardt1963algorithm} (see also \cite[Section~10.3]{nocedal2006numerical}) improves robustness by minimizing the regularized model $\min_d \tfrac12\|f(x_k)+J_k d\|^2+\tfrac{\mu_k}{2}\|d\|^2$ to stabilize poorly conditioned steps, a concept that was later analyzed and extended \cite{burke1995gauss, nesterov2007modified}. 
In practice, Levenberg-Marquardt underlies standard nonlinear least-squares libraries such as MINPACK and Ceres, and is widely used for curve fitting, parameter estimation, and geometric calibration across engineering and the physical sciences \cite{more1978levenberg,agarwal2012ceres,triggs2000bundle}.
These ideas are generalized into the powerful \textbf{trust-region framework} for smooth objectives $\min_x f(x)$, which at each iteration solves $\min_{\|d\|\le \Delta_k} m_k(d)$ and accepts/rejects by predicted/actual reduction agreement \cite{conn2000trust,fletcher2000practical}; early analyses established convergence even for nonsmooth objectives \cite{yuan1985conditions}.

In parallel, \textbf{proximal methods} emerged as a unifying abstraction for nonsmooth structures.
The classic \textbf{proximal point algorithm}, introduced by Martinet \cite{martinet1970regularisation} and developed by Rockafellar \cite{rockafellar1976monotone}, addresses $\min_x h(x)$ via regularized subproblems $\min_x h(x)+\tfrac{1}{2t}\|x-x^k\|^2$.
While often too expensive to apply directly, the proximal-point method serves as the conceptual basis for many practical splitting schemes 
(Douglas-Rachford, ADMM, Chambolle-Pock), which form the algorithmic core of modern large-scale signal processing, imaging, and distributed optimization methods \cite{eckstein1992douglas,boyd2011admm,combettes2011proximal, chambolle2011pdhg}.
This idea specializes to the influential \textbf{proximal-gradient} (forward--backward) method for convex ``smooth + simple'' objectives $\min_x f(x)+h(x)$, with accelerated variants (FISTA) improving worst-case rates \cite{beck2009fast,nesterov2013gradient}.
Proximal-gradient and its accelerated variants are central to modern data science, underpinning large-scale sparse learning (e.g., Lasso and structured sparsity) and $\ell_1$-type regularized inverse problems in imaging and compressed sensing \cite{tibshirani1996regression,daubechies2004iterative,chambolle2004algorithm,beck2009fast,combettes2011proximal}.
Comprehensive treatments and unifying perspectives appear in \cite{parikh2014proximal,drusvyatskiy2017proximal,bauschke2017convex}.

To achieve the fast local convergence rates (superlinear or quadratic) characteristic of second-order methods, proximal algorithms are extended to incorporate curvature information~\cite{lee2014proximal}. 
For additive convex composites $\min_x f(x)+h(x)$, \emph{proximal Newton} (PN) at $x_k$ solves the local problem $\min_d\ \nabla f(x_k)^\top d+\tfrac12 d^\top H_k d + h(x_k+d)$ with $H_k\approx\nabla^2 f(x_k)$; \emph{proximal quasi-Newton} (PQN) replaces $H_k$ by a curvature approximation $B_k$. 
Under standard local assumptions, such as strong convexity near the solution and a Lipschitz-continuous Hessian for PN, or Dennis-Moré-type conditions for PQN, these methods attain local quadratic and superlinear convergence rates, respectively, mirroring smooth Newton/Quasi-Newton behavior~\cite{lee2014proximal}. 

The proximal-gradient framework has been significantly extended to handle fully nonsmooth \textbf{nonconvex} problems $\min_x f(x)+h(x)$. 
A major theoretical breakthrough in the analysis of nonconvex algorithms came with the application of the \textbf{Kurdyka--\L{}ojasiewicz (K--\L{}) inequality}. The work of Attouch, Bolte, and Svaiter \cite{attouch2013convergence} provided a unified framework for proving convergence of a vast class of descent methods for nonsmooth, nonconvex functions that satisfy the K--\L{} property. This framework guarantees that if an algorithm produces a sequence that satisfies a sufficient decrease and a relative error condition, its iterates will converge to a single stationary point. 
Bolte, Sabach, and Teboulle \cite{bolte2014proximal} provided a prime example of this framework's power by using it to prove the convergence of the widely used Proximal Alternating Linearized Minimization (PALM) algorithm for block problems $\min_{x,y} f(x)+g(y)+H(x,y)$.
A complementary strategy replaces direct descent on $\min_x f(x)+h(x)$ with a smooth model. Themelis, Stella, and Patrinos \cite{themelis2018forward} introduced the \textbf{forward--backward envelope (FBE)}, an exact, strictly continuous penalty for the original problem. Performing a line search on the FBE yields global convergence guarantees for proximal-gradient schemes; moreover, when the search directions satisfy a Dennis--Moré condition (e.g., quasi-Newton updates), one obtains \emph{superlinear} convergence to a critical point \cite{themelis2018forward}.

Second-order information can also be incorporated in the nonconvex regularized setting. 
For fully nonconvex regularized objectives $\min_x f(x)+h(x)$, \emph{proximal quasi-Newton trust-region} (PQNTR) methods minimize the trust-region model
$\min_{\|d\|\le\Delta_k}\ \nabla f(x_k)^\top d+\tfrac12 d^\top B_k d + h(x_k+d)$,
using predicted/actual reduction ratio test and adaptive radii to secure global convergence to first-order points with worst-case complexity $\mathcal{O}(\varepsilon^{-2})$ \cite{aravkin2022proximal}. 
In nonsmooth regularized least squares $\min_x \tfrac12\|f(x)\|^2 + h(x)$, Levenberg-Marquardt (LM) ideas yield regularized subproblems
$\min_d\ \tfrac12\|f(x_k)+J_k d\|^2 + \tfrac{\lambda_k}{2}\|d\|^2 + h(x_k+d)$,
which deliver robust global behavior alongside fast local progress \cite{aravkin2024levenberg}.

The prox-linear (ProxDescent) algorithm extends the proximal gradient method to the more general composite setting, $\min_x \cvx(x) + \cout(\sinn(x))$ by, at each iteration, minimizing the local linearized model $\min_x \cvx(x)+\cout\!\big(\sinn(x_k)+\nabla \sinn(x_k)(x-x_k)\big)+\tfrac{1}{2t}\|x-x_k\|^2$. This approach has a long history, with foundational ideas appearing in the 1980s \cite{yuan1985conditions, burke1985descent,fletcher2009model}, and has been the subject of intense recent study and analysis \cite{cartis2011evaluation,lewis2016proximal,drusvyatskiy2019efficiency}. 
Prox-linear and related composite-model methods have been successfully applied to structured nonlinear inverse problems, including robust phase retrieval and quadratic sensing \cite{duchi2019solving}, robust blind deconvolution and bilinear sensing \cite{charisopoulos2019composite}, and nonconvex constrained trajectory generation via successive convexification \cite{elango2025continuous, kamath2025scvxgen, uzun2024successive}.
A unifying viewpoint for such composite problems is provided by the {Taylor-like model} framework~\cite{drusvyatskiy2021nonsmooth}, which studies algorithms driven by first-order models with two-sided quadratic approximation error.
Within this framework, a \emph{slope error bound} (slope EB) plays a central role. Under the two-sided quadratic model error bounds, a local slope EB for $F$ implies a \emph{step-size error bound} which in turn yields local linear convergence of model-based methods \cite{hu2016convergence,drusvyatskiy2018error,drusvyatskiy2021nonsmooth}. In particular, for broad classes of prox-regular functions, slope EB is equivalent (up to constants) to subdifferential error bounds and to the K--\L{} property with exponent $1/2$; in the convex (and smooth+convex composite) setting, quadratic growth, slope EB, subdifferential EB, KL exponent $1/2$, and the step-size EB are all equivalent notions of local regularity \cite{bolte2017error,drusvyatskiy2018error,drusvyatskiy2021nonsmooth}. Thus, in practice, any one of these conditions is sufficient to obtain a step-size EB and hence a local linear rate.

Second-order theory for the composite problem $\min_{x}\, \cvx(x)+\cout(\sinn(x))$ is also well developed. 
Classical local convergence theory for Newton and quasi-Newton methods for the composite problem $\min_{x}\, \cvx(x)+\cout(\sinn(x))$ uses second-order epi-derivatives to establish optimality conditions and sensitivity results \cite{rockafellar1989second,rockafellar1998variational}. 
More recent analyses frame the problem using generalized equations \cite{robinson1980strongly} and establish \emph{strong metric (sub)regularity} of the associated KKT mapping, particularly for piecewise linear convex functions \cite{cibulka2018strong}. 
The work of \cite{burke2018strong} extends this framework to piecewise linear-quadratic (PLQ) structures (see~\cite{rockafellar1998variational} and~\cite{aravkin2013sparse}) by augmenting the generalized equations approach with techniques from partial smoothness \cite{lewis2002active}, enabling Newton-style steps with superlinear or quadratic local rates under appropriate regularity \cite{burke2018line,burke2018strong}. 
When higher-order derivatives of the smooth component are accessible, higher-order composite methods can further sharpen oracle complexity bounds and have been shown to substantially reduce iteration counts in practice \cite{doikov2022high}.

Relative to the prox-linear method, prox-convex explicitly accommodates an additional smooth–convex composite term $\sout(\Cin(x))$, thereby allowing smooth couplings of convex (possibly nonsmooth) features.
Such terms naturally arise when signal temporal logic specifications are encoded via smooth robustness maps acting on convex predicate functions~\cite{donze2010robust,uzun2024optimization}. 
In powered descent guidance~\cite{uzun2025sequential}, implication-type specifications (state-triggered constraints) give rise to objectives and constraints of the form $s(R(x))$.
Likewise, eventual satisfaction requirements for convex range, field-of-view, and direction constraints in perception-aware motion planning for quadrotor flight are modeled through the same smooth–convex composite structure~\cite{uzun2025motion}.
Such structures have also been studied by Ochs et al. in a line of work on nonconvex, nonsmooth optimization~\cite{ochs2015iteratively,ochs2019nonsmooth}.
In \cite{ochs2015iteratively}, they propose iteratively reweighted majorization-minimization schemes for problems of the form $F_1(x)+F_2(G(x))$, with convex $F_1$, coordinatewise convex $G$, and nonconvex but coordinatewise nondecreasing $F_2$. A key contribution is the design of convex majorizers for popular nonconvex penalties (concave and convex-concave on $\mathbb{R}_+)$, leading to convex subproblems. This framework is flexible within those classes, but it relies on constructing explicit convex models for $F_2$, which becomes difficult outside the specific penalty families they treat. Moreover, the composite structure $F_2\circ G$ must still be amenable to efficient convex solvers.
In follow-up work, Ochs et al.~\cite{ochs2019nonsmooth} propose a unifying Bregman-proximal line-search framework: at each iteration, a convex model together with a Bregman distance generates a trial point, and an Armijo-type condition on the true objective enforces descent. This abstract scheme covers, among other cases, composite objectives $F_2\circ G$ with smooth, coordinatewise nondecreasing outer functions $F_2$ and convex $G$ (see, e.g.,~\cite[Ex.~36]{ochs2019nonsmooth}), while more general smooth but sign-indefinite $F_2$ are not treated in detail. Their analysis, based on the K--\L{} property, unifies a broad class of algorithms and guarantees convergence to critical points under standard assumptions (proper, closed objective and mild model conditions). The price of this generality is a mainly asymptotic theory, without explicit iteration-complexity bounds or quantitative links between model decrease, inexact model solves, and $\varepsilon$-stationarity.

\paragraph*{Main contributions:}
We propose prox-convex, an extension of the prox-linear method to
$F(x)=\cvx(x)+\cout(\sinn(x))+\sout(\Cin(x))$ that \emph{linearizes only the smooth maps} $\sinn,\sout$ while \emph{preserving the convex structure} in $\cvx,\cout$ and in the convex inner map $\Cin$. Compared with fully linearizing $\sout\circ\!\Cin$, this structure-preserving design is especially effective when $\Cin$ is \emph{highly curved, ill-conditioned, or nonsmooth}: the local model is tighter, accepted steps tend to be larger, the required proximal regularization is milder, and progress per solve improves.

\smallskip\noindent\emph{Adaptive globalization.}
At iterate $x_k$, we minimize a convex model plus a quadratic term $\tfrac12\|x-x_k\|_{Q_k}^2$ with $Q_k=\mu_kI+H_k^+\succ0$. The scalar proximal weight $\mu_k$ is chosen by a predicted/actual reduction ratio test that balances model error bounds and step length, yielding a monotone decrease on accepted steps while avoiding both instability and overly conservative progress \cite{conn2000trust,burke2018line}.

\smallskip\noindent\emph{Second-order injection.}
When available, we \emph{inject curvature} via $H_k^+\succeq0$ using $\mathcal C^2$ components of $\sinn$ and $\sout$, which tightens the model, improves conditioning, and accelerates local convergence \cite{burke1987second}.

\smallskip\noindent\emph{Theory backbone and guarantees.}
(i) Subdifferential regularity \cite{rockafellar1998variational,clarke1983optimization}: under the monotone-or-smooth condition, $F$ is regular so $\widehat\partial F=\partial F$ and stationarity is unambiguous (Theorem~\ref{thm:regular_mono_or_smooth}).
(ii) Convergence and complexity with an adaptive proximal metric $Q_k$: under mild Lipschitz/smoothness constants (Assumption~\ref{asm:TA}), a spectral threshold ensures acceptance, only finitely many rejections can occur, and the proximal metrics remain uniformly conditioned (Lemmas~\ref{lem:accept-Q}--\ref{lem:Qk-spectral-bounds}); accepted steps yield sufficient decrease and an $O(\varepsilon^{-2})$ bound (in accepted steps) to drive $\|\mathcal{G}_{Q_k}(x_k)\|\le\varepsilon$ (Theorem~\ref{thm:suff-dec-Q}).
(iii) Linear convergence from error bounds: building on the error-bound analysis of the prox-linear method with a fixed scalar proximal weight in \cite{drusvyatskiy2018error}, we extend the theory in three directions: we allow a matrix-valued, iteration-dependent proximal metric $Q_k$, work with an asymmetric quadratic model error bounds, and derive corresponding matrix-valued gradient inequalities for the metric prox-gradient. Under a \emph{metric} step-size error bound, this yields local $Q$-linear convergence of the function values for prox-convex (Theorem~\ref{thm:qlinear-pcx}); in the special case of a fixed scalar metric $Q_k=t^{-1}I$ and symmetric model error, our rate constant reduces to that of the prox-linear algorithm (Remark~\ref{rem:qlinear-scalar}).
(iv) Taylor-like framework and stationarity: combining quadratic model error bounds with boundedness (hence existence of cluster points) and asymptotic regularity $\|x_{k+1}-x_k\|\to 0$, and invoking the Taylor-like model framework~\cite{drusvyatskiy2021nonsmooth}, we guarantee that every cluster point of the accepted iterates is limiting-stationary (Theorem~\ref{thm:adaptive-stationary}). Under the regularity conditions of Theorem~\ref{thm:regular_mono_or_smooth}, this further implies Fréchet stationarity. The same framework also covers inexact subproblem solves and provides a model-decrease stopping rule.

\noindent\textbf{Notation.}
We use $\|\cdot\|$ for the Euclidean norm on vectors and the induced spectral norm on matrices, and $(a)_+ := \max\{a,0\}$. 

For a closed function $F:\mathbb{R}^m \to \mathbb{R}\cup\{+\infty\}$, its limiting slope at $\bar x$ is
\[
|\nabla F|(\bar x)\ :=\ \limsup_{x\to\bar x}\frac{\big(F(\bar x)-F(x)\big)_+}{\|x-\bar x\|},
\]
which reduces to $\|\nabla F(\bar x)\|$ in the $\mathcal C^1$ case.


\section{Prox-convex} \label{sec:pcx}
This section presents the prox-convex method. After collecting the standing assumptions, we define the convex local subproblem solved at each iteration and then describe an adaptive variable-metric choice $Q_k\succ0$, with optional curvature injection to improve conditioning and local progress.

\subsection{Assumptions}
We begin by stating a structural condition that rules out ``concave-through-a-kink'' pathologies (e.g.\ $x\mapsto -\|x\|_{1}$) and ensures regularity.
Informally, for each coordinate $i$ we require that either the outer map $\sout$ is locally nondecreasing in the $i$-th direction, or the inner map $\cin_i$ is smooth near the point of interest.
\begin{assumption}[Monotone-or-smooth per coordinate]
\label{ass:mono_or_smooth}
For each $\bar x\in\dom F$, there exists a neighborhood $W$ of $\Cin(\bar x)$ such that,
for every index $i\in\{1,\ldots,n\}$, the following holds:
if $\cin_i$ is not $\mathcal C^1$ on any neighborhood of $\bar x$, then
\[
  \nabla_i\sout(z)\ \ge\ 0 \quad \text{for all } z\in W.
\]
\end{assumption}
In particular, whenever $\nabla_i\sout(\Cin(\bar x))<0$, Assumption~\ref{ass:mono_or_smooth} guarantees that $\cin_i$ is $\mathcal C^1$ near $\bar x$, so $\nabla\cin_i(\bar x)$ is well-defined and can be used to build the convex model on coordinates where $\nabla_i\sout(\Cin(\bar x))<0$.

We next record mild Lipschitz and smoothness conditions on the constituent mappings, which will be used to control model errors, descent estimates, and conditioning of the proximal metric in the theoretical analysis.

\newpage

\begin{assumption}[Standing assumptions]\label{asm:TA}
Let $x_0\in \dom F$ be the starting point and define the initial level set
\[
\mathcal{X}_0 := \{x\in \dom F : F(x)\le F(x_0)\}.
\]
Assume $\mathcal{X}_0$ is nonempty and compact. Moreover, assume:
\begin{itemize}
    \item $\cout : \mathbb{R}^d \to \mathbb{R}$ is convex and $L_{\cout}$--Lipschitz;
    \item $\sinn : \mathbb{R}^m \to \mathbb{R}^d$ is $\mathcal{C}^1$ with $\beta_{\sinn}$--Lipschitz Jacobian;
    \item $\sout : \mathbb{R}^n \to \mathbb{R}$ is $\mathcal{C}^1$ with $\beta_{\sout}$--Lipschitz Jacobian;
    \item Each $\cin_i:\mathbb{R}^m\to\mathbb{R}$ is convex and $L_{\cin_i}$--Lipschitz, so that $\Cin:\mathbb{R}^m \to \mathbb{R}^n$ is $L_{\Cin}$--Lipschitz  (w.r.t.\ $\|\cdot\|_2$), where $L_{\Cin}:=\|(L_{\cin_1},\ldots,L_{\cin_n})\|_2$.
\end{itemize}

If there exists a point $x\in\mathcal{X}_0$ and an index $i$ such that $\nabla_i\sout(\Cin(x))<0$, then preserving convexity of the local model may require linearizing the corresponding component $\cin_i$ at such points. Therefore, we additionally assume:
\begin{itemize}
    \item $\|\nabla\sout(y)\|\le L_{\sout}$ for all $y \in \{\Cin(x) : F(x) \leq F(x_0)\}$;
    \item Any component $\cin_i$ that is ever linearized on the level set $\mathcal{X}_0$ is $\mathcal{C}^1$ with a $\beta_{\cin_i}$-Lipschitz gradient. Let $\mathcal I^-$ denote the set of all such indices, and define
    \[
      \beta_{\Cin} := \bigg(\sum_{i \in \mathcal I^-} \beta_{\cin_i}^2\bigg)^{1/2}.
    \]
\end{itemize}
\end{assumption}


\subsection{Method}
Given an iterate $x_k$, we construct a convex model of~\eqref{eq:def_F} by:
(i) linearizing the smooth inner map $\sinn$ inside the convex function $\cout$ (as in classical convex-composite modeling),
and (ii) linearizing $\sout$ while treating each convex coordinate map $\cin_i$ according to the sign of $\nabla_i\sout(\Cin(x_k))$ so that convexity is preserved. We define the convex model of \eqref{eq:def_F} at iteration $k$:
\begin{equation*}
    F(x; x_k) = \cvx(x)
    + \cout\big(\sinn(x_k) + \nabla \sinn(x_k)(x - x_k)\big)
    + \sout\big(\Cin(x_k)\big)
    + \sum_{i=1}^n \nabla_i \sout\big(\Cin(x_k)\big)\, \Phi_i(x; x_k),
\end{equation*}
where
\begin{equation*}
    \Phi_i(x; x_k) =
    \begin{cases}
        \cin_i(x) - \cin_i(x_k) 
            & \text{if } \nabla_i \sout(\Cin(x_k)) \geq 0, \\[4pt]
        \nabla \cin_i(x_k)^\top (x - x_k) & \text{if } \nabla_i \sout(\Cin(x_k)) < 0.
    \end{cases}
\end{equation*}
We denote the set of \emph{linearized channels} at $x_k$ as $\mathcal I_k^- \;:=\; \big\{\,i\in\{1,\dots,n\} : \nabla_i\sout(\Cin(x_k))<0\,\big\}.$

\begin{remark} \label{rem:negative-channel}
If $\nabla_i\sout(\Cin(x_k)) \ge 0$, then the contribution
$
x\ \mapsto\ \nabla_i\sout(\Cin(x_k)) \big(\cin_i(x)-\cin_i(x_k)\big)
$
is convex (a nonnegative scalar times a convex function), so we keep the convex structure intact in the subproblem.
If $\nabla_i\sout(\Cin(x_k))<0$, then
$
x\ \mapsto\ \nabla_i\sout(\Cin(x_k))\big(\cin_i(x)-\cin_i(x_k)\big)
$
is \emph{concave}; to preserve convexity of the model we replace this concave term by its affine first-order expansion
$
x\ \mapsto\ \nabla_i\sout(\Cin(x_k))\,\nabla\cin_i(x_k)^\top(x-x_k),
$
which is a global convex \emph{majorant} of that concave channel. Thus $x\mapsto F(x;x_k)$ is convex.
\end{remark}

Next, we define the proximal model at iteration $k$ as
\begin{equation*}
    F_{Q_k}(x; x_k) = F(x; x_k) + \frac{1}{2}\|x - x_k\|_{Q_k}^2, \text{ where $Q_k = \mu_k I + H_k^+$}
\end{equation*}
is a positive definite matrix that may vary across iterations. Each $Q_k$ combines a scalar proximal weight $\mu_k>0$ with a curvature approximation $H_k^+ \succeq 0$ that captures local second-order information. Since $F(\cdot;x_k)$ is convex and $Q_k\succ0$, the function $x\mapsto F_{Q_k}(x;x_k)$ is $\lambda_{\min}(Q_k)$-strongly convex.

The prox-convex update at $x_k$ is the unique minimizer of the strongly convex model $F_{Q_k}(x;x_k)$:
\begin{equation*}
    x_{k+1} = \arg\min_{x} F_{Q_k}(x; x_k).
\end{equation*}
This subproblem can be solved efficiently either by interior-point methods
\cite{karmarkar1984new,nesterov1994interior,mehrotra1992implementation}
or by first-order primal-dual schemes
\cite{chambolle2011pdhg,boyd2011admm,yu2022proportional}, depending on the structure of $\cvx$ and the $\cin_i$.

\begin{remark}[Choosing inner vs.\ outer linearization]
For composite functions, when one mapping is convex and both mappings are $\mathcal{C}^1$, linearizing only the nonconvex component while keeping the convex component intact yields a convex subproblem; under standard smoothness and Lipschitz assumptions, Proposition~\ref{prop:errors} shows that this single-component linearization is uniformly tighter than linearizing both components.

When both mappings are convex and $\mathcal{C}^1$, a practical guideline is to keep the more curved component exact and linearize the less curved one. Even in the special case where the outer mapping is nondecreasing (so the composite is convex; see \cite[(3.10)]{boyd2004convex}), a linearized formulation may still be preferable for compatibility with standard convex solvers.
\end{remark}


\subsection{Adaptive proximal metric}
This subsection specifies the proximal metric $Q_k\succ0$.
We describe a practical construction that combines a scalar damping term with optional curvature information, and we show how to adapt the scalar weight online using predicted-versus-actual decrease.

\subsubsection{Construction of the Hessian Approximation}
\label{subsec:hessian_construction}
When $\sinn$ or $\sout$ are twice continuously differentiable, one can construct $H_k^+$ from their Hessians to enhance curvature adaptation. If such Hessian information is unavailable or computationally intractable, $H_k^+$ can be safely set to zero. 

\paragraph{Curvature for $\cout(\sinn(x))$ (inner-only).}
Since $\cout$ is kept \emph{exact} in the subproblem, we do not include any
outer curvature. When $\sinn$ is $\mathcal C^2$ and
tractable, we add the inner curvature weighted by a subgradient of $\cout$:
\[
H_{\sinn, k}
\;:=\;
\sum_{j=1}^d y_j\,\nabla^2\sinn_j(x_k),
\qquad
y \in \partial\cout\!\big(\sinn(x_k)\big).
\]
If $\cout$ is smooth, we take $y=\nabla\cout(\sinn(x_k))$. If $\sinn$ is not
$\mathcal C^2$ or Hessians are intractable, we set
$H_{\sinn,k}=0$.

\paragraph{Curvature for $\sout(\Cin(x))$ (outer pullback + inner compensation on linearized channels).}
When $\sout$ is $\mathcal C^2$, we include the outer pullback; when, in addition, a linearized channel $i\in\mathcal I_k^-$ has $\cin_i\in\mathcal C^2$ and a tractable Hessian, we also include its inner curvature (which is negative semidefinite once weighted by the negative outer component):
\[
H_{\sout, k}
\;:=\;
\underbrace{G_{\Cin}(x_k)^\top\,
\nabla^2\sout(\Cin(x_k))\,G_{\Cin}(x_k)}_{\text{outer pullback (if $\sout\in\mathcal C^2$)}}
\;+\;
\underbrace{\sum_{i\in\mathcal I_k^-}
\big[\nabla\sout(\Cin(x_k))\big]_i\,\nabla^2\cin_i(x_k)}_{\text{inner compensation on linearized channels (if $\cin_i\in\mathcal C^2$)}},
\]
where
\[
G_{\Cin}(x_k)
:=
\begin{bmatrix}
g_{1,k}^\top \\[-2pt]
\vdots \\[-2pt]
g_{n,k}^\top
\end{bmatrix},
\quad
g_{i,k} \in \partial \cin_i(x_k).
\]
When each $\cin_i$ is smooth, $g_{i,k} = \nabla \cin_i(x_k)$ and $G_{\Cin}(x_k) = J_{\Cin}(x_k)$; for nonsmooth $\cin_i$, the stacked subgradients form an affine first-order model of $\Cin(x)$. For channels with $[\nabla\sout(\Cin(x_k))]_i \ge 0$ we keep $\cin_i$ \emph{exact} in the subproblem and therefore do not add its (positive) curvature. If $\sout$ is not $\mathcal C^2$ or the required Hessians are unavailable, we set $H_{\sout,k}=0$.

\medskip
\noindent\textbf{Final metric and bounds.}
We assemble the curvature as
\[
H_k^+ \;:=\; \Pi_{\mathbb{S}_+}\!\Big( H_{\sinn, k} \;+\; H_{\sout, k} \Big), \quad Q_k = \mu_k I + H_k^+,
\]
where $\Pi_{\mathbb{S}_+}$ denotes the projection onto the positive semidefinite cone.

\begin{remark}[Justification of the Hessian construction]\label{rem:curv-just}
Under the local second-order smoothness conditions assumed for the Hessian-based model (in particular, bounded/Lipschitz Hessians; see Assumption~\ref{asm:H_smth}), the unprojected Hessian model is locally third-order accurate. After PSD projection, the only second-order discrepancy is the nonpositive quadratic ``projection gap'', while the remaining error terms are third order or higher (see Corollary~\ref{cor:joint-proj}).
\end{remark}


\subsubsection{Adaptive scalar proximal weight} \label{subsec:adp_weight}
The local error bounds of each model are governed by two varying quantities, the outer and inner curvatures of the composite function, so any \emph{static} proximal weight inevitably becomes either unstable (too small in highly curved regions) or overly conservative (too large in flat regions). Our adaptive prox-convex scheme therefore tunes the proximal metric $Q_k=\mu_kI+H_k^+$ from iteration to iteration using the agreement between \emph{predicted} and \emph{actual} decrease: when the model is accurate, we decrease the scalar proximal weight $\mu_k$ to take bolder steps; when it is not, we increase $\mu_k$ to restore reliability, implicitly tracking the unknown local constants without ever computing them.

Whenever $\sinn$ or $\sout$ is $\mathcal{C}^2$ and Hessian evaluation is tractable, we enrich $Q_k$ with PSD curvature blocks $H_k^+\succeq0$, which tightens the model, improves conditioning, and sharpens the local decrease rates; otherwise we fall back to first-order damping while preserving $Q_k\succ0$. This combination yields robust step acceptance with larger effective steps in benign regions and safe contraction near strong inner curvature, translating into faster and more reliable convergence in practice.

The trust-region idea dates back to the 1970s \cite{powell1970new} and is classically used to update the radius \cite[Ch.~6]{conn2000trust}; analogous ideas update scalar proximal weights \cite{cartis2011evaluation}. 
Here we adopt an adaptive scheme that scales the \emph{proximal metric} in the quadratic term rather than a scalar alone. The resulting adaptive prox-convex algorithm is presented in Algorithm~\ref{alg:adaptive-pcx}.

\begin{algorithm}[htbp]
\caption{Adaptive prox-convex with proximal metric update}
\label{alg:adaptive-pcx}
\begin{algorithmic}[1]
\State Input $x_0$, $\mu_0>0$, $\mu_{\min}>0$, thresholds $0<\alpha_1<\alpha_2<1$, factors $\nu_{\rm inc}>1>\nu_{\rm dec}>0$
\For{$k=0,1,2,\dots$}
  \State Build $H_k^+\succeq0$;
  \Repeat \Comment{inner loop (retries on rejection)}
    \State Set $Q_k\gets \mu_k I+H_k^+$;
    \State $x_k^+ = \arg\min_x \big\{F(x;x_k)+\tfrac12\|x-x_k\|_{Q_k}^2\big\}$;
    \State $\Pred_k=F(x_k)-F_{Q_k}(x_k^+;x_k)$;\quad $\Act_k=F(x_k)-F(x_k^+)$ \Comment{predicted and actual decreases}
    \If{$\Pred_k=0$ or $\|Q_k(x_k-x_k^+)\|\le \epsilon_{\mathrm{term}}$} \Comment{termination criteria}
        \State \textbf{return} $x_k$
    \EndIf
    \State $\rho_k=\Act_k/\Pred_k$  \Comment{acceptance ratio}
    \If{$\rho_k<\alpha_1$} \Comment{poor agreement}
        \State $\mu_k\gets \nu_{\rm inc}\mu_k$ 
    \EndIf
  \Until{$\rho_k\ge \alpha_1$} \Comment{acceptance criteria}
  \State $x_{k+1}\gets x_k^+$
    \If{$\rho_k>\alpha_2$} \Comment{excellent agreement}
      \State $\mu_{k+1}\gets \max\{\mu_{\min},\nu_{\rm dec}\mu_k\}$ 
    \Else \Comment{moderate agreement}
      \State $\mu_{k+1}\gets \mu_k$ 
    \EndIf
\EndFor
\end{algorithmic}
\end{algorithm}


\section{Theoretical Analysis}
This section establishes global convergence and complexity guarantees for prox-convex.
We proceed as follows.
(i) We first prove subdifferential regularity $\widehat\partial F(\bar x)=\partial F(\bar x)$ (Theorem~\ref{thm:regular_mono_or_smooth}) under the monotone-or-smooth condition (Assumption~\ref{ass:mono_or_smooth}).
(ii) We state mild smoothness/Lipschitz assumptions (Assumption~\ref{asm:TA}).
(iii) For the actual prox-convex model, we derive two-sided quadratic error bounds with $L_U:=L_{\cout}\beta_{\sinn}+L_{\Cin}^2\beta_{\sout}$ and $L_L:=L_U+L_{\sout}\beta_{\Cin}$ (Lemma~\ref{lem:model_two_sided}).
(iv) We analyze the acceptance mechanism: a simple spectral threshold guarantees acceptance (Lemma~\ref{lem:accept-Q}); at most finitely many consecutive rejections occur (Lemma~\ref{lem:finite-reject-Q}); and the adaptive proximal metrics $Q_k$ enjoy uniform spectral bounds (Lemma~\ref{lem:Qk-spectral-bounds}).
(v) As a consequence, accepted steps achieve a sufficient decrease and an $O(\varepsilon^{-2})$ bound (in accepted steps) for reducing the norm of the metric prox-gradient $\|\mathcal{G}_{Q_k}(x_k)\|$ below $\varepsilon$ (Theorem~\ref{thm:suff-dec-Q}).
(vi) We then establish \emph{linear convergence} under local error bounds: extending the error-bound analysis of the prox-linear method with a fixed \emph{scalar} proximal term in~\cite{drusvyatskiy2018error} to our \emph{time-varying matrix} metric $Q_k$ and asymmetric two-sided model constants $(L_L,L_U)$, we derive matrix-valued gradient inequalities for the metric prox-gradient and show that a \emph{metric} step-size error bound implies local $Q$-linear decay of function values (Theorem~\ref{thm:qlinear-pcx}); in the fixed-metric scalar case $Q_k=t^{-1}I$ with symmetric model constants, the resulting rate constant reduces to that of the prox-linear algorithm (Remark~\ref{rem:qlinear-scalar}).
(vii) Finally, combining our quadratic model error bounds with sufficient decrease (hence asymptotic regularity $\|x_{k+1}-x_k\|\to 0$) and boundedness of the accepted iterates, and invoking the Taylor-like model framework of Drusvyatskiy--Ioffe--Lewis~\cite{drusvyatskiy2021nonsmooth}, we prove that every cluster point of the accepted iterates is limiting-stationary (Theorem~\ref{thm:adaptive-stationary}). Under the regularity conditions of Theorem~\ref{thm:regular_mono_or_smooth}, this further implies Fr\'echet stationarity. The same framework also justifies a model-decrease termination rule and robustness to inexact subproblem solves.


\subsection{Subdifferential Regularity}
We begin by ensuring that the various subdifferential notions agree at the points of interest, so that ``stationarity'' can be stated unambiguously for limit points of the iterates. For a proper, lower semicontinuous function $F$, one typically distinguishes the Fr\'echet (regular) subdifferential $\widehat\partial F$, the limiting (Mordukhovich) subdifferential $\partial F$, and, for locally Lipschitz $F$, the Clarke subdifferential $\partial^{\circ}F$ \cite{clarke1983optimization,rockafellar1998variational}. These coincide for smooth or convex functions, but in the nonsmooth, nonconvex setting the inclusions
\[
\widehat\partial F(\bar x)\ \subset\ \partial F(\bar x)\ \subset\ \partial^{\circ}F(\bar x)
\]
may be strict. A simple example is $F(x)=-\|x\|_1$ at $x=0$, where $\widehat\partial F(0)=\emptyset$, $\partial F(0)=\{-1,1\}$ in one dimension (or the set of sign vectors in higher dimensions), and $\partial^{\circ}F(0)=[-1,1]$.
The following result shows that, under our monotone-or-smooth condition, $F$ is \emph{subdifferentially regular} at $\bar x$, so that $\widehat\partial F(\bar x)=\partial F(\bar x)$; in particular, all standard first-order stationarity conditions agree at $\bar x$.

\begin{theorem}[Subdifferential regularity]\label{thm:regular_mono_or_smooth}
Suppose Assumption~\ref{ass:mono_or_smooth} holds at $\bar x$.
Let $W$ be the corresponding neighborhood of $\Cin(\bar x)$ and define
\[
\mathcal I_{\bar{x}}\ :=\ \Big\{\,i\in\{1,\dots,n\} : \nabla_i\sout(z)\ge 0\ \text{for all } z\in W\,\Big\},
\]
so that for every $i\notin \mathcal I_{\bar{x}}$ the channel $\cin_i$ is $\mathcal{C}^1$ on a neighborhood of $\bar x$.
Then $F$ is subdifferentially regular at $\bar x$; in particular, its Fr\'echet and limiting
subdifferentials coincide: $\widehat\partial F(\bar x)=\partial F(\bar x)$. Moreover,
\[
\partial F(\bar x)
\;=\;
\partial\cvx(\bar x)
\;+\;
\nabla\sinn(\bar x)^{\!\top}\,\partial\cout(\sinn(\bar x))
\;+\;
\sum_{i\in \mathcal I_{\bar{x}}} \nabla_i\sout(\Cin(\bar x))\,\partial\cin_i(\bar x)
\;+\;
\sum_{i\notin \mathcal I_{\bar{x}}} \nabla_i\sout(\Cin(\bar x))\,\nabla\cin_i(\bar x),
\]
where the sums are Minkowski sums.
\end{theorem}

\begin{proof}
By convex analysis, $\cvx$ is regular with $\widehat\partial\cvx(\bar x)=\partial\cvx(\bar x)$.
Since $\cout$ is convex and $\sinn$ is $\mathcal{C}^1$, the basic chain rule in \cite[Thm.~10.6]{rockafellar1998variational} yields
\[
\widehat\partial(\cout(\sinn(\bar x)))=\partial(\cout(\sinn(\bar x)))=\nabla\sinn(\bar x)^{\!\top}\,\partial\cout(\sinn(\bar x)).
\]

Since $\sout$ is $\mathcal C^1$, it is regular at $\Cin(\bar x)$ and $\partial\sout(\Cin(\bar x))=\{\nabla\sout(\Cin(\bar x))\}$. 
Moreover, each $\cin_i$ is convex finite, hence $\Cin$ is locally Lipschitz (strictly continuous) near $\bar x$. 
Thus the {equality} case of the extended chain rule \cite[Thm.~10.49]{rockafellar1998variational} applies provided $y^\top\Cin$ is regular for $y=\nabla\sout(\Cin(\bar x))$. 
By Assumption~\ref{ass:mono_or_smooth}, for each $i$ either $y_i\ge 0$ so $y_i\cin_i$ is convex (regular), or $\cin_i$ is $\mathcal C^1$ near $\bar x$ so $y_i\cin_i$ is smooth (regular). 
Hence, $y^\top\Cin$ is regular and
\[
\partial(\sout(\Cin(\bar x))) 
= D^\star\Cin(\bar x)\big[\partial\sout(\Cin(\bar x))\big]
= \sum_{i\in \mathcal I_{\bar{x}}} \nabla_i\sout(\Cin(\bar x))\,\partial\cin_i(\bar x)
  + \sum_{i\notin \mathcal I_{\bar{x}}} \nabla_i\sout(\Cin(\bar x))\,\nabla\cin_i(\bar x),
\]
where $D^\star\Cin(\bar x)$ denotes the (limiting) coderivative of the mapping $\Cin$ at $\bar x$. Therefore, $\sout\circ\Cin$ is regular at $\bar x$.

Finally, $\cout\circ\sinn$ and $\sout\circ\Cin$ are locally Lipschitz, hence 
$\partial^\infty(\cout\circ\sinn)(\bar x)=\partial^\infty(\sout\circ\Cin)(\bar x)=\{0\}$; 
the horizon qualification in \cite[Cor.~10.9]{rockafellar1998variational} is therefore automatic. 
Since $\cvx$, $\cout\circ\sinn$, and $\sout\circ\Cin$ are regular at $\bar x$, the addition rule yields
\[
\partial F(\bar x)=\partial\cvx(\bar x)+\partial(\cout(\sinn(\bar x)))+\partial(\sout(\Cin(\bar x))),
\]
and $F$ is regular at $\bar x$, i.e., $\widehat\partial F(\bar x)=\partial F(\bar x)$.
\end{proof}


\subsection{Global descent and complexity bound}
We show that the prox-convex model satisfies explicit two-sided quadratic approximation bounds, and therefore fits the Taylor-like model framework.
\begin{lemma}[Quadratic Model Error]\label{lem:model_two_sided}
For all $x$, the model error is bounded by:
\begin{align} \label{eq:two_sided}
-
\frac{L_L}{2}
\|x-x_k\|^2
\;\le\; F(x)-F(x;x_k) \;\le\; 
\frac{L_U}{2}
\|x-x_k\|^2.
\end{align}
where $L_L := L_U + L_{\sout}\beta_{\Cin}$ and $L_U := L_{\cout}\beta_{\sinn} + L_{\Cin}^2\beta_{\sout}$.
\end{lemma}

\begin{proof}
The proof strategy is to decompose the error $F(x) - F(x; x_k)$ into its two main composite parts and then establish both an upper and a lower bound for each component. Let $d_x := x-x_k$. By construction, the terms involving $\cvx(x)$ cancel, leaving:
\begin{align*}
F(x)-F(x;x_k)
&=\underbrace{\left[\cout(\sinn(x))-\cout\big(\sinn(x_k)+\nabla\sinn(x_k)d_x\big)\right]}_{\text{Term (I): Error from } \cout \circ \sinn} 
+
\underbrace{\left[\sout(\Cin(x))-\sout(\Cin(x_k))-\nabla\sout(\Cin(x_k))^\top\Phi(x;x_k)\right]}_{\text{Term (II): Error from } \sout \circ \Cin}.
\end{align*}

\paragraph{Bounding Term (I)}
This term represents the error from linearizing the inner map $\sinn$. Let $\Delta_{\sinn} := \sinn(x)-\sinn(x_k)-\nabla\sinn(x_k)d_x$. Since $\nabla\sinn$ is $\beta_{\sinn}$-Lipschitz, we have $\|\Delta_{\sinn}\|\le \frac{\beta_{\sinn}}{2}\|d_x\|^2$. The $L_{\cout}$-Lipschitz continuity of $\cout$ then provides a symmetric quadratic bound:
\[
-\frac{L_{\cout}\beta_{\sinn}}{2}\|d_x\|^2 \;\le\; \mathrm{(I)} \;\le\; \frac{L_{\cout}\beta_{\sinn}}{2}\|d_x\|^2.
\]

\paragraph{Bounding Term (II)}
This term represents the error from the prox-convex approximation of the $\sout \circ \Cin$ composition. Let $\Delta_{\Cin} := \Cin(x)-\Cin(x_k)$. We can rewrite Term (II) to separate the error from the smoothness of $\sout$ and the error from the model's specific design:
\[
\mathrm{(II)} = \underbrace{\left[\sout(\Cin(x))-\sout(\Cin(x_k))-\nabla\sout(\Cin(x_k))^\top\Delta_{\Cin}\right]}_{\text{Smoothness Error}} + \underbrace{\nabla\sout(\Cin(x_k))^\top(\Delta_{\Cin}-\Phi(x;x_k))}_{\text{Model Design Error}}.
\]
The standard inequality for the $\beta_{\sout}$-smooth function $\sout$ bounds the \emph{Smoothness Error}. 
Furthermore, by construction of our model, the \emph{Model Design Error} is always non-positive. 
For each component $i$:
\[
\cin_i(x) - \cin_i(x_k) - \Phi_i(x; x_k) = 
\begin{cases}
0 & \text{if } \nabla_i \sout\big(\Cin(x_k)\big) \geq 0 \\
\cin_i(x)-\cin_i(x_k)-\nabla\cin_i(x_k)^\top d_x & \text{otherwise}.
\end{cases}
\]
By convexity of $\cin_i$, $\cin_i(x)-\cin_i(x_k)-\nabla\cin_i(x_k)^\top d_x \geq 0$. This implies that:
\[
\nabla\sout(\Cin(x_k))^\top(\Delta_{\Cin} - \Phi(x; x_k)) = \sum_{i=1}^n \nabla_i \sout\big(\Cin(x_k)\big) \big(\cin_i(x) - \cin_i(x_k) - \Phi_i(x; x_k)\big) \leq 0.
\]
This immediately gives the upper bound:
\[
\mathrm{(II)} \le \frac{\beta_{\sout}}{2}\|\Delta_{\Cin}\|^2 + 0 \le \frac{\beta_{\sout}L_{\Cin}^2}{2}\|d_x\|^2.
\]
For the lower bound, we use the other side of the smoothness inequality and the Cauchy-Schwarz inequality:
\[
\mathrm{(II)} \ge -\frac{\beta_{\sout}}{2}\|\Delta_{\Cin}\|^2 - \|\nabla\sout(\Cin(x_k))\| \cdot \|\Delta_{\Cin}-\Phi(x;x_k)\|.
\]
The error vector $\Delta_{\Cin}-\Phi$ has non-zero components only if $\nabla_i \sout\big(\Cin(x_k)\big) < 0$. For these components, our additional assumption provides a bound on the linearization error: $|\cin_i(x)-\cin_i(x_k)-\nabla\cin_i(x_k)^\top d_x| \le \frac{\beta_{\cin_i}}{2}\|d_x\|^2$. The norm of the full error vector is therefore bounded:
\[
\|\Delta_{\Cin}-\Phi\| \le \frac{\beta_{\Cin}}{2}\|d_x\|^2.
\]
Substituting this into the lower bound for Term (II) and using $\|\nabla\sout\| \le L_{\sout}$ and $\|\Delta_{\Cin}\| \le L_{\Cin}\|d_x\|$, we get:
\[
\mathrm{(II)} \ge -\frac{\beta_{\sout}L_{\Cin}^2}{2}\|d_x\|^2 - \frac{L_{\sout}\beta_{\Cin}}{2}\|d_x\|^2.
\]

Adding the bounds for (I) and (II) gives 
$-\frac{L_L}{2}\|d_x\|^2 \le F(x)-F(x;x_k) \le \frac{L_U}{2}\|d_x\|^2,$
where $L_L := L_U + L_{\sout}\beta_{\Cin}$ and $L_U := L_{\cout}\beta_{\sinn} + L_{\Cin}^2\beta_{\sout}$.
\end{proof}

\begin{remark}[Model constants]\label{rem:abstract-LU-LL-compact}
As will be shown, $L_U$ controls how large $Q_k$ must be to ensure acceptance (and hence directly governs the effective step-size), whereas $L_L$ is needed to translate model progress into stationarity guarantees and to establish the linear convergence rate.

\emph{Key points.}
(i) We linearize a channel $\cin_i$ only when the outer directional influence is negative (the contribution is locally concave); the affine term is a global majorant and preserves convexity.
(ii) Curvature of such linearized inner channels enters only through an additive term in $L_L$ (via the design-error bound) and does \emph{not} affect $L_U$; hence step-size selection depends solely on $L_U$.
(iii) If no channel is linearized (i.e., $\mathcal I^-=\emptyset$), then the bound is symmetric and $L_L=L_U$.
\end{remark}


We connect model error bounds to the predicted/actual reduction ratio test, proving that a minimal amount of proximal metric guarantees acceptance.

\begin{lemma}[Sufficient condition for step acceptance]\label{lem:accept-Q}
Let $Q_k\succ0$ and write $\sigma_k:=\lambda_{\min}(Q_k)$. If
\[
\sigma_k\ \ge\ \frac{L_U}{\,2-\alpha_1\,},
\]
then the trial point $x_k^+$ is accepted, i.e., $\rho_k:=\Act_k/\Pred_k\ge \alpha_1$.
\end{lemma}

\begin{proof}
By strong convexity of $x\mapsto F_{Q_k}(x;x_k)$ and optimality of $x_k^+$,
\[
\Pred_k\ =\ F(x_k)-F_{Q_k}(x_k^+;x_k)\ \ge\ \tfrac12\|x_k^+-x_k\|_{Q_k}^2.
\]
By the upper side of Lemma~\ref{lem:model_two_sided} at $x=x_k^+$,
\[
\Act_k\ =\ F(x_k)-F(x_k^+)
\ \ge\ F(x_k)-F(x_k^+;x_k)-\tfrac{L_U}{2}\|x_k^+-x_k\|^2
\ =\ \Pred_k+\tfrac12\|x_k^+-x_k\|_{Q_k}^2-\tfrac{L_U}{2}\|x_k^+-x_k\|^2.
\]
Using $\|u\|_{Q_k}^2\ge \sigma_k\|u\|^2$,
\[
\frac{\Act_k}{\Pred_k}\ \ge\ 1+\frac{\sigma_k-L_U}{\sigma_k}
\ =\ 2-\frac{L_U}{\sigma_k}\ \ge\ \alpha_1
\]
whenever $\sigma_k\ge L_U/(2-\alpha_1)$.
\end{proof}
Thus, once the trial curvature exceeds the model smoothness threshold $L_U$, discounted by the acceptance ratio $(2-\alpha_1)$, the predicted/actual reduction ratio test always passes.


We bound the number of consecutive rejections via geometric growth of the scalar proximal weight, ensuring progress of the inner loop.
\begin{lemma}[Finite rejections]\label{lem:finite-reject-Q}
Assume the setting of Lemma~\ref{lem:accept-Q}. Let $Q_k=\mu_k I+H_k^+$ with
$H_k^+\succeq0$ fixed within the inner loop and update $\mu_k\leftarrow\nu_{\rm inc}\mu_k$ on each rejection, with $\nu_{\rm inc}>1$. Writing
$\psi:=L_U/(2-\alpha_1)$, after at most
\[
N_k\ \le\ \Big\lceil \log\!\big(\psi/\mu_k\big)\big/\log(\nu_{\rm inc})\Big\rceil_+
\]
rejections we have $\lambda_{\min}(Q_k)\ge\psi$, so the step is accepted. In particular, if $\mu_k\ge \psi$ then $N_k=0$. Here $\lceil\cdot\rceil_+:=\max\{0,\lceil\cdot\rceil\}$.
\end{lemma}

\begin{proof}
Within the inner loop the curvature block $H_k^+$ is fixed and only the scalar
$\mu$ is updated. After $j$ consecutive rejections,
\[
Q^{(j)}=\mu^{(j)} I + H_k^+,\qquad \mu^{(j)}=\nu_{\rm inc}^{\,j}\,\mu_k.
\]
Since $H_k^+\succeq0$, $\lambda_{\min}\!\big(Q^{(j)}\big)\ge\mu^{(j)}$. By Lemma~\ref{lem:accept-Q}, acceptance is guaranteed once
$\lambda_{\min}(Q^{(j)})\ge \psi$, which is implied by $\mu^{(j)}\ge \psi$. Using $\mu^{(j)}=\nu_{\rm inc}^{\,j}\mu_k$ gives
\[
j\ \ge\ \frac{\log(\psi/\mu_k)}{\log(\nu_{\rm inc})}.
\]
Taking the smallest integer $j$ satisfying the inequality and clipping at $0$ yields the claim. If $\mu_k\ge\psi$ then $N_k=0$ and the first trial is accepted. 
\end{proof}


We show that the adaptive proximal metrics remain uniformly well-conditioned, which is crucial for descent and complexity.
\begin{lemma}[Uniform spectral bounds for $Q_k$]\label{lem:Qk-spectral-bounds}
Run Algorithm~\ref{alg:adaptive-pcx} with $0<\alpha_1<\alpha_2<1$, $\nu_{\rm inc}>1$, $\nu_{\rm dec}\in(0,1)$, and $\mu_{\min}>0$.
Let $Q_k=\mu_k I+H_k^+$ where $H_k^+$ is constructed as in Sec.~\ref{subsec:hessian_construction}. Then there exist constants $0<\underline q\le \overline q<\infty$ (independent of $k$) such that
\begin{equation}\label{eq:spectral-bounds}
\underline q\,I\ \preceq\ Q_k\ \preceq\ \overline q\,I\qquad(\forall k),
\end{equation}
with
\[
\underline q:=\min\{\mu_0,\mu_{\min}\},\qquad
\overline q:=\max\!\Big\{\mu_0,\ \frac{\nu_{\rm inc}L_U}{\,2-\alpha_1\,}\Big\}+\overline h,
\]
and the curvature cap
\[
\boxed{\ \overline h\ :=\ \beta_{\sinn}L_{\cout}\ +\ \beta_{\sout}L_{\Cin}^2\ +\ L_{\sout}\,\beta_{\Cin} \ }.
\]
\end{lemma}

\begin{proof}
\emph{Step 1 (Spectral bound for $H_k^+$).}
Because the projection $\Pi_{\mathbb{S}_+}$ is nonexpansive,
$\|H_k^+\|\le \|H_{\sinn,k}+H_{\sout,k}\|$. Assumption~\ref{asm:TA} and standard Lipschitz–Hessian
bounds yield
\[
\Big\|\sum_{j} y_j\,\nabla^2\sinn_j(x_k)\Big\|
\ \le\ \beta_{\sinn}\,\|y\|_2
\ \le\ \beta_{\sinn}\,L_{\cout},
\qquad
\big\|G_{\Cin}(x_k)^\top\nabla^2\sout(\Cin(x_k))\,G_{\Cin}(x_k)\big\|
\ \le\ \beta_{\sout}\,L_{\Cin}^2.
\]
If each linearized channel $i\in\mathcal I_k^- \subseteq \mathcal I^-$ satisfies
$\cin_i\in\mathcal C^1$ with $\|\nabla^2\cin_i(x)\|\le \beta_{\cin_i}$, then using $\|\nabla\sout(\Cin(x_k))\|\le L_{\sout}$, we have
\[
\Bigg\|\sum_{i\in\mathcal I_k^-}\big[\nabla\sout(\Cin(x_k))\big]_i
\,\nabla^2\cin_i(x_k)\Bigg\|
\ \le\ \|\nabla\sout(\Cin(x_k))\|\,\beta_{\Cin}
\ \le\ L_{\sout}\,\beta_{\Cin}.
\]
Therefore,
\[
0\ \le\ \lambda_{\min}(H_k^+)\ \le\ \lambda_{\max}(H_k^+)\ \le\ \overline h
:= \beta_{\sinn}L_{\cout}+\beta_{\sout}L_{\Cin}^2+L_{\sout}\beta_{\Cin}.
\]

\emph{Step 2 (Lower spectral bound for $Q_k$).}
By the update rules in Algorithm~\ref{alg:adaptive-pcx},
\[
\mu_{k+1}=
\begin{cases}
\nu_{\rm inc}\mu_k & \text{if rejected},\\
\max\{\mu_{\min},\,\nu_{\rm dec}\mu_k\} & \text{if accepted with excellent agreement},\\
\mu_k & \text{if accepted with moderate agreement},
\end{cases}
\]
so by induction $\mu_k\!\ge\! \min\{\mu_0,\mu_{\min}\}$ for all $k$. Since $H_k^+\!\succeq\! 0$,
we have $\lambda_{\min}(Q_k)\!\ge\! \mu_k\!\ge\! \underline q$, hence $\underline q\,I\preceq Q_k$.

\emph{Step 3 (Upper spectral bound for $Q_k$).}
Let $\psi:=L_U/(2-\alpha_1)$. By Lemma~\ref{lem:finite-reject-Q}, within any inner loop the step is accepted once $\mu\ge \psi$, and during rejections we only multiply $\mu$ by $\nu_{\rm inc}$. Therefore the scalar at acceptance obeys
\(
\mu^{\rm acc}\le \nu_{\rm inc}\psi=\nu_{\rm inc}L_U/(2-\alpha_1)
\),
and during that inner loop $\mu$ never exceeds this cap.
Across outer iterations, acceptance either keeps or shrinks $\mu$, hence
\[
\mu_k\ \le\ \max\!\Big\{\mu_0,\ \frac{\nu_{\rm inc}L_U}{2-\alpha_1}\Big\}\qquad(\forall k).
\]
Finally,
\(
\lambda_{\max}(Q_k)\le \mu_k+\lambda_{\max}(H_k^+)
\le \max\{\mu_0,\nu_{\rm inc}L_U/(2-\alpha_1)\}+\overline h
=: \overline q
\).
Thus $Q_k\preceq \overline q\,I$, completing \eqref{eq:spectral-bounds}.
\end{proof}


With acceptance and conditioning in place, we derive a per-step decrease and a global $O(\varepsilon^{-2})$ bound in accepted steps. Throughout, let $\mathcal S\subset\mathbb N$ denote the set of \emph{accepted} indices produced by Algorithm~\ref{alg:adaptive-pcx}, and write $\Delta_F:=F(x_0)-\inf F<\infty$. 
\begin{theorem}[Global descent and $O(\varepsilon^{-2})$ complexity on accepted steps]\label{thm:suff-dec-Q}
Let $F$ be bounded below and let $\{x_k\}_{k\in\mathcal S}$ be the accepted iterates of Algorithm~\ref{alg:adaptive-pcx} with ratio threshold $\alpha_1\in(0,1)$ and metrics $Q_k$ satisfying Lemma~\ref{lem:Qk-spectral-bounds}. Then for every $k\in\mathcal S$,
\begin{equation}\label{eq:adaptive-sd-final}
F(x_k)-F(x_{k+1})
\ \ge\ \frac{\alpha_1}{2}\,\|x_{k+1}-x_k\|_{Q_k}^2
\ \ge\ \frac{\alpha_1\,\underline q}{2}\,\|x_{k+1}-x_k\|^2.
\end{equation}
Consequently $\sum_{k\in\mathcal S}\|x_{k+1}-x_k\|_{Q_k}^2<\infty$, hence $\|x_{k+1}-x_k\|_{Q_k}\to0$. By Lemma~\ref{lem:Qk-spectral-bounds}, $\|\mathcal{G}_{Q_k}(x_k)\|=\|Q_k(x_k-x_{k+1})\|\le \overline q\,\|x_{k+1}-x_k\|\le (\overline q/\sqrt{\underline q})\,\|x_{k+1}-x_k\|_{Q_k}\to0$ along $\mathcal S$. Moreover, with
$\Delta_F:=F(x_0)-\inf F$ and for the first $N$ accepted indices $\mathcal S_N$,
\begin{equation}\label{eq:residual-avg-min}
\frac{1}{N}\sum_{k\in\mathcal S_N}\|\mathcal{G}_{Q_k}(x_k)\|^2
\ \le\ \frac{2\,\overline q\,\Delta_F}{\alpha_1\,N},
\qquad
\min_{k\in\mathcal S_N}\|\mathcal{G}_{Q_k}(x_k)\|
\ \le\ \sqrt{\frac{2\,\overline q\,\Delta_F}{\alpha_1\,N}}.
\end{equation}
Hence $N=O(\varepsilon^{-2})$ accepted steps suffice to reach $\min_{k\in\mathcal S_N}\|\mathcal{G}_{Q_k}(x_k)\|\le\varepsilon$.
\end{theorem}

\begin{proof}
For accepted $k$, $\Act_k\ge \alpha_1\,\Pred_k$. 
Strong convexity of $x\mapsto F_{Q_k}(x;x_k)$ and optimality of $x_{k+1}$ gives
$\Pred_k\ge \tfrac12\|x_{k+1}-x_k\|_{Q_k}^2$, yielding the first inequality in \eqref{eq:adaptive-sd-final}; the second follows from $Q_k\succeq \underline q I$. Summing over accepted indices and using $\inf F>-\infty$ gives $\sum_{k\in\mathcal S}\|x_{k+1}-x_k\|_{Q_k}^2<\infty$. For \eqref{eq:residual-avg-min}, note $Q_k^{-1}\succeq \overline q^{-1}I$, so
\[
F(x_k)-F(x_{k+1})\ \ge\ \frac{\alpha_1}{2}\,\|\mathcal{G}_{Q_k}(x_k)\|_{Q_k^{-1}}^2
\ \ge\ \frac{\alpha_1}{2\overline q}\,\|\mathcal{G}_{Q_k}(x_k)\|^2.
\]
Sum over $\mathcal S_N$, divide by $N$, and use $\min\le$ average.
\end{proof}

\begin{remark}[Effect of rejections]
If at most $j_{\max}$ consecutive rejections occur between accepted steps (cf.\ Lemma~\ref{lem:finite-reject-Q}), then at most $(1+j_{\max})N$ subproblem solves are needed to obtain $N$ \emph{accepted} steps. Thus, the overall evaluation complexity remains $O(\varepsilon^{-2})$ up to a constant factor.
\end{remark}

\begin{remark}[Unconditional decrease without the ratio test]
Combining the same inequalities used in the proof of Lemma~\ref{lem:accept-Q} yields
\[
F(x_k)-F(x_{k+1})
\ \ge\ \tfrac12\!\left(2-\frac{L_U}{\sigma_k}\right)\!\|x_{k+1}-x_k\|_{Q_k}^2,
\qquad \sigma_k:=\lambda_{\min}(Q_k).
\]
Thus, whenever $\sigma_k>\tfrac12L_U$, we have descent even \emph{without} invoking the ratio test. 
\end{remark}


\subsection{Linear rate of convergence}
Conceptually, this subsection mirrors the prox-linear analysis with a fixed \emph{scalar} proximal weight in \cite{drusvyatskiy2018error}, but in a more general setting: our prox-convex scheme uses an \emph{iteration-dependent matrix metric} $Q_k$ and relies on the \emph{asymmetric} two-sided model bounds of Lemma~\ref{lem:model_two_sided}, with distinct constants $L_U$ and $L_L$ rather than a single symmetric constant. We first derive matrix-valued gradient inequalities for the metric prox-gradient $\mathcal{G}_{Q_k}$ (a matrix analogue of \cite[Lem.~5.1]{drusvyatskiy2018error}). Combined with a \emph{metric} step-size error bound in the sense of \cite[Def.~5.4]{drusvyatskiy2018error}, these estimates yield local $Q$-linear convergence of function values following the template of \cite[Thm.~5.5]{drusvyatskiy2018error}. In the fixed-metric scalar case $Q_k=t^{-1}I$ and $L_L=L_U$, the resulting rate constant reduces to that of the prox-linear algorithm.

\begin{lemma}[Gradient inequality for prox-convex]
\label{lem:grad-ineq-pcx}
Let $x_{k+1}=\arg\min_x\{F(x;x_k)+\tfrac12\|x-x_k\|_{Q_k}^2\}$ and define the metric prox-gradient mapping
\[
\mathcal{G}_{Q_k}(x_k)\ :=\ Q_k(x_k-x_{k+1})\qquad(\text{so }Q_k^{-1}\mathcal{G}_{Q_k}(x_k)=x_k-x_{k+1}).
\]
Then for all $y\in\mathbb{R}^m$ the following hold.

\medskip\noindent
\begin{equation}\label{eq:gi-model}
F(y;x_k)\ \ge\ F_{Q_k}(x_{k+1};x_k)\ +\ \langle \mathcal{G}_{Q_k}(x_k),\,y-x_k\rangle\ +\ \tfrac12\langle Q_k^{-1}\mathcal{G}_{Q_k}(x_k),\,\mathcal{G}_{Q_k}(x_k)\rangle .
\end{equation}

\noindent
If Lemma~\ref{lem:model_two_sided} holds with constants $L_U,L_L$ and 
$\sigma_k:=\lambda_{\min}(Q_k)$, then
\begin{equation}\label{eq:gi-function}
F(y)\ \ge\ F(x_{k+1})\ +\ \langle \mathcal{G}_{Q_k}(x_k),\,y-x_k\rangle\ +\Big(1-\tfrac{L_U}{2\sigma_k}\Big)\,\langle Q_k^{-1}\mathcal{G}_{Q_k}(x_k),\,\mathcal{G}_{Q_k}(x_k)\rangle\ -\ \tfrac{L_L}{2}\,\|y-x_k\|^2 .
\end{equation}
In particular, setting $y=x_k$ gives the sufficient decrease estimate
\begin{equation}\label{eq:gi-decrease}
F(x_k)\ -\ F(x_{k+1})\ \ge\ \Big(1-\tfrac{L_U}{2\sigma_k}\Big)\,\langle Q_k^{-1}\mathcal{G}_{Q_k}(x_k),\,\mathcal{G}_{Q_k}(x_k)\rangle
\ =\ \tfrac12\Big(2-\tfrac{L_U}{\sigma_k}\Big)\,\|x_{k+1}-x_k\|_{Q_k}^2 .
\end{equation}
When $Q_k=t^{-1}I$ and $L_L=L_U=L$ (the symmetric-error bounds case of \cite{drusvyatskiy2018error}), \eqref{eq:gi-function} reduces to
\[
F(y)\ \ge\ F(x_{k+1})+\langle \mathcal{G}_{t^{-1}}(x_k),y-x_k\rangle+\tfrac{t}{2}(2-L t)\,\|\mathcal{G}_{t^{-1}}(x_k)\|^2-\tfrac{L}{2}\|y-x_k\|^2,
\]
which matches \cite[Lem.~5.1]{drusvyatskiy2018error} in the scalar-stepsize case.
\end{lemma}

\begin{proof}
For brevity let $d_k:=x_{k+1}-x_k$ and $\mathcal{G}_{Q_k} := \mathcal{G}_{Q_k}(x_k)=Q_k(x_k-x_{k+1})=-Q_k d_k$.
Consider the strongly convex function $F_{Q_k}(x;x_k)$. Since $x_{k+1}$ is its unique minimizer, for all $y$,
\[
F_{Q_k}(y;x_k)\ \ge\ F_{Q_k}(x_{k+1};x_k)\ +\ \tfrac12\|y-x_{k+1}\|_{Q_k}^2.
\]
Therefore, we get
\[
F(y;x_k)+\tfrac12\|y-x_k\|_{Q_k}^2\ \ge\ F_{Q_k}(x_{k+1};x_k)\ +\ \tfrac12\|y-x_{k+1}\|_{Q_k}^2.
\]
Write $y-x_{k+1}=(y-x_k)+Q_k^{-1}\mathcal{G}_{Q_k}$ and expand the last term:
\[
\|y-x_{k+1}\|_{Q_k}^2
=\|y-x_k\|_{Q_k}^2 + 2\langle \mathcal{G}_{Q_k},\,y-x_k\rangle + \langle Q_k^{-1}\mathcal{G}_{Q_k},\mathcal{G}_{Q_k}\rangle.
\]
Substituting this into the previous inequality and cancelling $\tfrac12\|y-x_k\|_{Q_k}^2$ from both sides yields \eqref{eq:gi-model}.

For the function values, Lemma~\ref{lem:model_two_sided} gives, for all $y$,
\[
F(y)\ \ge\ F(y;x_k)-\tfrac{L_L}{2}\|y-x_k\|^2,
\]
Combining this with \eqref{eq:gi-model} yields
\begin{equation}
\begin{aligned} \label{eq:comb_lw}
F(y)
&\ \ge\ F_{Q_k}(x_{k+1};x_k)
  +\langle \mathcal{G}_{Q_k},y-x_k\rangle
  +\tfrac12\langle Q_k^{-1}\mathcal{G}_{Q_k},\mathcal{G}_{Q_k}\rangle
  -\tfrac{L_L}{2}\|y-x_k\|^2
\end{aligned}
\end{equation}
Lemma~\ref{lem:model_two_sided} also gives, at $x_{k+1}$,
\[
F(x_{k+1};x_k)\ \ge\ F(x_{k+1})-\tfrac{L_U}{2}\|d_k\|^2.
\]
Thus
\[
F_{Q_k}(x_{k+1};x_k)
=F(x_{k+1};x_k)+\tfrac12\|d_k\|_{Q_k}^2
\ \ge\ F(x_{k+1})+\tfrac12\|d_k\|_{Q_k}^2 - \tfrac{L_U}{2}\|d_k\|^2.
\]
Since $Q_k\succeq\sigma_k I$, we have $\|d_k\|^2\le\sigma_k^{-1}\|d_k\|_{Q_k}^2$, so
\[
F_{Q_k}(x_{k+1};x_k)
\ \ge\ F(x_{k+1})+\Big(\tfrac12-\tfrac{L_U}{2\sigma_k}\Big)\|d_k\|_{Q_k}^2
=F(x_{k+1})+\Big(\tfrac12-\tfrac{L_U}{2\sigma_k}\Big)\langle Q_k^{-1}\mathcal{G}_{Q_k},\mathcal{G}_{Q_k}\rangle.
\]
Combining this with \eqref{eq:comb_lw} yields
\[
\begin{aligned}
F(y)
&\ \ge\ F(x_{k+1})
  +\langle \mathcal{G}_{Q_k},y-x_k\rangle
  +\Big(1-\tfrac{L_U}{2\sigma_k}\Big)\langle Q_k^{-1}\mathcal{G}_{Q_k},\mathcal{G}_{Q_k}\rangle
  -\tfrac{L_L}{2}\|y-x_k\|^2,
\end{aligned}
\]
which is \eqref{eq:gi-function}. Setting $y=x_k$ removes the last term and gives \eqref{eq:gi-decrease}.
\end{proof}


\begin{definition}[Metric step-size error bound]\label{def:eb}
Let $\{x_k\}$ be generated by Algorithm~\ref{alg:adaptive-pcx}, and let $x^\star$ be a limit point of $\{x_k\}$.
We say that $\{x_k\}$ satisfies a \emph{metric step-size error bound at $x^\star$} if there exist a constant
$\kappa>0$, a neighborhood $\mathcal U$ of $x^\star$, and an index $K$ such that, for all $k\ge K$,
\begin{equation}\label{eq:metric-EB}
  x_k\in\mathcal U,
  \qquad
  \|x_k-x^\star\|
  \ \le\ \kappa\,\big\|\mathcal{G}_{Q_k}(x_k)\big\|.
\end{equation}
\end{definition}

\noindent\textit{Relation to existing error-bound notions and implications for prox-convex.}
When $Q_k=(1/t)I$, condition \eqref{eq:metric-EB} reduces exactly to the scalar step-size error bound of \cite[Def.~5.4]{drusvyatskiy2018error}. More generally, if the metrics are uniformly well-conditioned, $\sigma_{\min}I\preceq Q_k\preceq\sigma_{\max}I$, then the metric step-size EB \eqref{eq:metric-EB} is equivalent, up to constants depending only on $(\sigma_{\min},\sigma_{\max})$, to the corresponding scalar step-size EB by norm equivalence. Combined with the Taylor-like model error bounds of $F_{Q_k}(\cdot;x_k)$, any local slope EB (or, in the convex/composite case, any of its equivalent forms such as a K--\L{} exponent $1/2$ or quadratic growth) therefore yields a step-size error bound of the form~\eqref{eq:metric-EB} \cite[Thm.~3.5]{drusvyatskiy2021nonsmooth}, which is exactly the ingredient we use below to establish $Q$-linear convergence of the prox-convex iterates.

\begin{theorem}[Q-linear convergence under a metric step-size error bound]
\label{thm:qlinear-pcx}

Assume Definition~\ref{def:eb} holds at a limit point $x^\star$ of the sequence $\{x_k\}$ generated by
Algorithm~\ref{alg:adaptive-pcx}, and the uniform lower spectral bound satisfies $\underline q>L_U/(2-\alpha_1)$.
Set $F^\star := F(x^\star)$. Then all steps are accepted (Lemma~\ref{lem:accept-Q}), and
\begin{equation}\label{eq:qlinear-rate}
  F(x_{k+1})-F^\star\ \le\ q^\star\,\big(F(x_k)-F^\star\big)
  \qquad\text{for all }k\ge K,
\end{equation}
with contraction factor
\begin{equation}\label{eq:q-factor}
  q^\star\ :=\
  1-\min\!\left\{1,\,
  \frac{\,2-\tfrac{L_U}{\underline q}\,}
           {\overline q\big(2+L_L\big)\,\bar\kappa^{\,2}}
  \right\}
  \ \in[0,1),
  \qquad
  \bar\kappa:=\max\{1,\kappa\}.
\end{equation}
Consequently, $\{F(x_k)\}$ converges $Q$-linearly to $F(x^\star)$.
\end{theorem}

\begin{proof}
Let
$\mathcal{G}_{Q_k}:=\mathcal{G}_{Q_k}(x_k)$ and
$\sigma_k:=\lambda_{\min}(Q_k)$.
Applying Lemma~\ref{lem:grad-ineq-pcx} with $y=x^\star$, and using
$\langle Q_k^{-1}\mathcal{G}_{Q_k},\mathcal{G}_{Q_k}\rangle
\ge \overline q^{-1}\|\mathcal{G}_{Q_k}\|^2$ and $\sigma_k\ge \underline q$, gives
\begin{equation}\label{eq:value-gap-local}
  F(x_{k+1})-F(x^\star)
  \ \le\
  \|\mathcal{G}_{Q_k}\|\,\|x_k-x^\star\|
  +\frac{L_L}{2}\|x_k-x^\star\|^2
  +\frac{L_U/\underline{q}-2}{2\bar{q}}\,\|\mathcal{G}_{Q_k}\|^2.
\end{equation}
Define
\[
  \kappa_k\ :=\ \max\bigg\{1,\ \frac{\|x_k-x^\star\|}{\|\mathcal{G}_{Q_k}\|}\bigg\},
\]
(with the convention that the ratio is $0$ when $\mathcal{G}_{Q_k}=0$).
Then $\|x_k-x^\star\|\le\kappa_k\|\mathcal{G}_{Q_k}\|$ and
$\|x_k-x^\star\|^2\le\kappa_k^2\|\mathcal{G}_{Q_k}\|^2$, so from
\eqref{eq:value-gap-local} we deduce
\[
  F(x_{k+1})-F(x^\star)
  \ \le\
  \Bigg[ \kappa_k+\frac{L_L}{2}\kappa_k^2
        +\frac{L_U/\underline{q}-2}{2\bar{q}}\Bigg]\,
  \|\mathcal{G}_{Q_k}\|^2
  \ \le\
  \Bigg[\bigg(1+\frac{L_L}{2}\bigg)\kappa_k^2
        +\frac{L_U/\underline{q}-2}{2\bar{q}}\Bigg]\,
  \|\mathcal{G}_{Q_k}\|^2.
\]

On the other hand, the sufficient decrease estimate
\eqref{eq:gi-decrease} and the bound
$\langle Q_k^{-1}\mathcal{G}_{Q_k},\mathcal{G}_{Q_k}\rangle\ge\bar{q}^{-1}\|\mathcal{G}_{Q_k}\|^2$ yield
\[
  F(x_k)-F(x_{k+1})
  \ \ge\
  \Big(1-\frac{L_U}{2\underline{q}}\Big)\,\langle Q_k^{-1}\mathcal{G}_{Q_k},\mathcal{G}_{Q_k}\rangle
  \ \ge\
  \frac{2-\tfrac{L_U}{\underline{q}}}{2\bar{q}}\,\|\mathcal{G}_{Q_k}\|^2.
\]
Combining the last two displays, we obtain
\[
  \frac{F(x_{k+1})-F^\star}{F(x_k)-F(x_{k+1})}
  \ \le\
  \frac{
    \big(1+\tfrac{L_L}{2}\big)\kappa_k^2
      +\tfrac{L_U/\underline{q}-2}{2\bar{q}}}
    {\tfrac{2-\tfrac{L_U}{\underline{q}}}{2\bar{q}}}
  \ =\
  \frac{2\bar{q}\big(1+\tfrac{L_L}{2}\big)\kappa_k^2}
       {2-\tfrac{L_U}{\underline{q}}} - 1.
\]
Hence
\[
  F(x_{k+1})-F^\star
  \ \le\
  \hat{q}^\star_k\,\big(F(x_k)-F(x_{k+1})\big),
  \qquad
  \hat{q}^\star_k:=
  \frac{2\bar{q}\big(1+\tfrac{L_L}{2}\big)\kappa_k^2}
       {2-\tfrac{L_U}{\underline{q}}} - 1.
\]
Rewriting this as
$F(x_{k+1})-F^\star \le \frac{\hat{q}^\star_k}{1+\hat{q}^\star_k}\big(F(x_k)-F^\star\big)$, we see that the per-iteration contraction factor is
\[
  q_k^\star\ :=\ \frac{\hat{q}^\star_k}{1+\hat{q}^\star_k}
  \ =\
  1-\frac{2-\tfrac{L_U}{\underline{q}}}
          {2\bar{q}\big(1+\tfrac{L_L}{2}\big)\kappa_k^2}.
\]

Now, invoke the metric step-size error bound
\eqref{eq:metric-EB}. For all sufficiently large $k$, $\|x_k-x^\star\| \le \kappa\|\mathcal{G}_{Q_k}\|$, and therefore
$\kappa_k\le\bar\kappa:=\max\{1,\kappa\}$. Hence,
for all sufficiently large $k$,
\[
  q_k^\star\ \le\
  1-\frac{2-\tfrac{L_U}{\underline q}}
          {2\,\overline q\big(1+\tfrac{L_L}{2}\big)\,\bar\kappa^{\,2}}
  \ =\
  1-\frac{\,2-\tfrac{L_U}{\underline q}\,}
           {\overline q\big(2+L_L\big)\,\bar\kappa^{\,2}}
  \ \le\
  1- \min \Bigg\{1, \ \frac{\,2-\tfrac{L_U}{\underline q}\,}
           {\overline q\big(2+L_L\big)\,\bar\kappa^{\,2}} \Bigg\}
  \ =:\ q^\star\ \in[0,1).
\]
This gives \eqref{eq:qlinear-rate} and completes the proof.
\end{proof}

\begin{remark}[Comparison with prox-linear] \label{rem:qlinear-scalar}
In the scalar case $Q_k=t^{-1}I$, the prox-convex step coincides with the
prox-linear step with parameter $t$ from
\cite{drusvyatskiy2018error}, and
$\mathcal{G}_{Q_k}(x_k)=\mathcal{G}_{t^{-1}}(x_k)$. The metric error bound
\eqref{eq:metric-EB} then specializes exactly to the scalar step-size error bound
of \cite[Def.~5.4]{drusvyatskiy2018error}:
\[
  \|x_k-x^\star\|
  \ \le\ \kappa\,\|\mathcal{G}_{t^{-1}}(x)\|.
\]
In this scalar setting we have $\underline q = \overline q = t^{-1}$
and $L_U=L_L=L$ (the symmetric-error bounds constant of \cite{drusvyatskiy2018error}), so the
general contraction factor \eqref{eq:q-factor} becomes
\[
  q^\star\ \le\
  1-\frac{t(2-L t)}{(2+L)\,\bar\kappa^{\,2}},
  \qquad
  \bar\kappa=\max\{1,\kappa\}.
\]
Choosing the standard stepsize $t=1/L$ yields
\[
  q^\star\ \le\ 1-\frac{1}{L(2+L)\,\bar\kappa^{\,2}},
\]
which \emph{exactly matches} the prox-linear linear rate constant in
\cite[Thm.~5.5]{drusvyatskiy2018error}.
\end{remark}


\subsection{Taylor-like framework and stationarity}
We now complete the argument using the Taylor-like model framework of Drusvyatskiy--Ioffe--Lewis~\cite{drusvyatskiy2021nonsmooth}.
Their key estimate~\cite[Thm.~3.1]{drusvyatskiy2021nonsmooth} (proved via Ekeland’s variational principle~\cite{ekeland1974variational}) shows that a uniform quadratic model-error bound yields a nearby point with small stationarity measure (see also~\cite[Cor.~3.2]{drusvyatskiy2021nonsmooth}).
Building on this, \cite[Cor.~3.3]{drusvyatskiy2021nonsmooth} implies subsequence stationarity under exact model minimization, asymptotic regularity $\|x_{k+1}-x_k\|\to 0$, and boundedness.
In our setting, these requirements are ensured by: (i) two-sided quadratic model error bounds (Lemma~\ref{lem:model_two_sided}); (ii) sufficient decrease and asymptotic regularity along accepted steps (Theorem~\ref{thm:suff-dec-Q}); and (iii) boundedness of the accepted subsequence.
Therefore, every cluster point of the accepted iterates is first-order stationary; combined with Theorem~\ref{thm:regular_mono_or_smooth}, stationarity holds for both the limiting and Fr\'echet subdifferentials.
\begin{theorem}[Stationary cluster points]\label{thm:adaptive-stationary}
Let $\{x_k\}_{k\in\mathcal S}$ be the accepted iterates of Algorithm~\ref{alg:adaptive-pcx}.
Assume $F$ is proper, closed, bounded below, $\{x_k\}_{k\in\mathcal S}$ is bounded, and Lemma~\ref{lem:model_two_sided} holds.
Then every cluster point $\bar x$ of $\{x_k\}_{k\in\mathcal S}$ is first-order stationary: $0\in\partial F(\bar x)$.
\end{theorem}

\begin{proof}
By Theorem~\ref{thm:suff-dec-Q}, we have $\|x_{k+1}-x_k\|\to0$ along $\mathcal S$.
Moreover, Lemma~\ref{lem:model_two_sided} and the spectral upper bound $\lambda_{\max}(Q_k)\le \overline q$ imply that for all $y$ and all accepted indices $k\in\mathcal S$,
\[
\big|F_{Q_k}(y;x_k)-F(y)\big|
\le
\frac{L_L+\lambda_{\max}(Q_k)}{2}\,\|y-x_k\|^2
\le
\omega(\|y-x_k\|),
\qquad
\omega(r):=\tfrac12(L_L+\overline q)r^2,
\]
so the growth condition of \cite[Cor.~3.3]{drusvyatskiy2021nonsmooth} holds.
Since $\{x_k\}_{k\in\mathcal S}$ is bounded, it has a cluster point $\bar x$; take a subsequence $x_{k_i}\to \bar x$.
By (accepted-step) monotonicity and boundedness below, $F(x_k)$ converges along $\mathcal S$, hence $F(x_{k_i})$ converges as well; in particular, $(x_{k_i},F(x_{k_i}))\to (\bar x,F(\bar x))$.
Applying \cite[Cor.~3.3]{drusvyatskiy2021nonsmooth} yields that $\bar x$ is stationary for $F$, i.e., $0\in \partial F(\bar x)$.
\end{proof}

\begin{remark}[Inexact solves and model-decrease termination]
By the Taylor-like framework of \cite{drusvyatskiy2021nonsmooth}, our two-sided model bounds imply robustness to \emph{inexact} subproblem solves. If
\[
F_{Q_k}(x_{k+1};x_k)\ \le\ \inf_y F_{Q_k}(y;x_k)+\varepsilon_k,
\]
then \cite[Cor.~5.2]{drusvyatskiy2021nonsmooth} guarantees a nearby point $\hat x_k$ with
\[
|\nabla F|(\hat x_k)\ \le\ \sqrt{12\,\eta_k\,\varepsilon_k}
\;+\;3\eta_k\|x_{k+1}-x_k\|,
\qquad
\eta_k:=L_L+\lambda_{\max}(Q_k).
\]
Consequently, if $\varepsilon_k\to 0$ and $\|x_{k+1}-x_k\|\to0$ along accepted steps, the stationarity conclusions remain valid.

Moreover, the \emph{model decrease}
\[
\Delta_k:=F(x_k)-\inf_y F_{Q_k}(y;x_k)
\]
is a practical stopping criteria: \cite[Cor.~5.5]{drusvyatskiy2021nonsmooth} yields, for any $\eta\ge \eta_k$, the existence of $\hat x$ with
\[
|\nabla F|(\hat x)\ \le\ \sqrt{12\,\eta\,\Delta_k}.
\]
\end{remark}

\bibliographystyle{plainurl}
\bibliography{main}


\section{Appendix}
\subsection{Model errors}
In composite problems, the effective level of linearization hinges on the roles of the outer and inner maps. If the outer map is smooth and the inner components are convex, linearizing the \emph{outer} layer preserves convexity; if the outer is convex and the inner is smooth but nonconvex, linearizing the \emph{inner} layer is natural. When \emph{both} $\sout$ and the $\cin_i$ are convex and $\mathcal{C}^1$–smooth, the choice is less obvious. We therefore compare three models: (i) \emph{full} linearization $F^{\rm all}$, (ii) \emph{inner-only} linearization $F^{\rm in}$, and (iii) \emph{outer-only} linearization $F^{\rm out}$, and derive clean error bounds that reveal the dominant constants.

To compare the three linearization levels on equal footing, we strengthen Assumption~\ref{asm:TA} for this subsection as follows: in addition to the original standing assumptions, we assume $\|\nabla\sout(y)\|\le L_{\sout}$ there, and that every $\cin_i$ is $\mathcal C^1$ with $\beta_{\cin_i}$–Lipschitz gradient. We then write
$\beta_{\Cin}:=\big(\sum_{i=1}^n \beta_{\cin_i}^2\big)^{1/2}$.

\begin{proposition}[Modeling error bounds]\label{prop:errors}
Fix $x_k$, write $d_x:=x-x_k$, $\Cin_k:=\Cin(x_k)$, $J_k:=J_{\Cin}(x_k)$, and $\Delta_{\Cin}:=\Cin(x)-\Cin_k$.
Define
\[
\begin{aligned}
F^{\rm all}(x;x_k)&:=\sout(\Cin_k)+\nabla\sout(\Cin_k)^\top(J_k d_x),\\
F^{\rm in}(x;x_k)&:=\sout(\Cin_k+J_k d_x),\\
F^{\rm out}(x;x_k)&:=\sout(\Cin_k)+\nabla\sout(\Cin_k)^\top(\Delta_{\Cin}),
\end{aligned}
\qquad
E^{(\cdot)}(x;x_k):=\sout(\Cin(x))-F^{(\cdot)}(x;x_k).
\]
Under the assumptions above,
\[
|E^{\rm all}(x;x_k)| \le \Big(\tfrac{\beta_{\sout}L_{\Cin}^2}{2}+\tfrac{L_{\sout}\beta_{\Cin}}{2}\Big)\|d_x\|^2,\quad
|E^{\rm in}(x;x_k)| \le \tfrac{L_{\sout}\beta_{\Cin}}{2}\|d_x\|^2,\quad
|E^{\rm out}(x;x_k)| \le \tfrac{\beta_{\sout}L_{\Cin}^2}{2}\|d_x\|^2.
\]
\end{proposition}

\begin{proof}
Let the inner Taylor remainder be
\[
e_{\Cin}(x):=\Cin(x)-\big(\Cin_k+J_k d_x\big).
\]
By component-wise $\beta_{\cin_i}$–smoothness and stacking,
\begin{equation}\label{eq:eR}
\|e_{\Cin}(x)\|\le \frac{\beta_{\Cin}}{2}\|d_x\|^2.
\end{equation}
Lipschitzness of $\Cin$ yields
\begin{equation}\label{eq:LCin}
\|\Delta_{\Cin}\|\le L_{\Cin}\|d_x\|,\qquad \|J_k d_x\|\le L_{\Cin}\|d_x\|.
\end{equation}
For $\sout$ we use the descent lemma and a gradient-bound shift:
\begin{equation}\label{eq:smooth+shift}
\big|\sout(u)-\sout(v)-\nabla\sout(v)^\top(u-v)\big|\le \tfrac{\beta_{\sout}}{2}\|u-v\|^2,\quad
\big|\sout(u)-\sout(v)\big|\le L_{\sout}\|u-v\|.
\end{equation}

\emph{All:}
\[
\begin{aligned}
|E^{\rm all}|
&= \big|\sout(\Cin_k+\Delta_{\Cin})-\sout(\Cin_k)-\nabla\sout(\Cin_k)^\top(J_k d_x)\big|\\
&\le \underbrace{\big|\sout(\Cin_k+\Delta_{\Cin})-\sout(\Cin_k+J_k d_x)\big|}_{\le L_{\sout}\|e_{\Cin}(x)\|}
+ \underbrace{\big|\sout(\Cin_k+J_k d_x)-\sout(\Cin_k)-\nabla\sout(\Cin_k)^\top(J_k d_x)\big|}_{\le \frac{\beta_{\sout}}{2}\|J_k d_x\|^2}\\
&\le \tfrac{L_{\sout}\beta_{\Cin}}{2}\|d_x\|^2+\tfrac{\beta_{\sout}L_{\Cin}^2}{2}\|d_x\|^2,
\end{aligned}
\]
by \eqref{eq:eR}, \eqref{eq:LCin}, \eqref{eq:smooth+shift}.

\emph{In:}
\[
|E^{\rm in}| = \big|\sout(\Cin_k+\Delta_{\Cin})-\sout(\Cin_k+J_k d_x)\big|
\le L_{\sout}\|e_{\Cin}(x)\|
\le \tfrac{L_{\sout}\beta_{\Cin}}{2}\|d_x\|^2.
\]

\emph{Out:}
\[
|E^{\rm out}| = \big|\sout(\Cin_k+\Delta_{\Cin})-\sout(\Cin_k)-\nabla\sout(\Cin_k)^\top\Delta_{\Cin}\big|
\le \tfrac{\beta_{\sout}}{2}\|\Delta_{\Cin}\|^2
\le \tfrac{\beta_{\sout}L_{\Cin}^2}{2}\|d_x\|^2.
\]
\end{proof}

The outer-only model is favored when $\beta_{\sout}L_{\Cin}^2\ll L_{\sout}\beta_{\Cin}$ (strong inner curvature but gently curved outer map), while the inner-only model is preferable in the opposite regime; the full linearization combines both penalties and is typically dominated. In our algorithm, prox-convex couples naturally with an \emph{adaptive proximal} metric and, when tractable, \emph{second-order curvature blocks} from $\mathcal{C}^2$ components, further tightening local models and improving the observed rate.

\subsection{Bounds for Hessian-Augmented Models}\label{app:hess-bounds}

In this subsection, in addition to Assumption~\ref{asm:TA}, we assume the following higher-order smoothness conditions to provide bounds for Hessian-augmented models supporting Section~\ref{subsec:hessian_construction}. 

\begin{assumption}[Additional smoothness assumptions] \label{asm:H_smth}
In the level set $\mathcal{X}_0$, we assume:
\begin{itemize}
    \item $\cout : \mathbb{R}^d \to \mathbb{R}$ is $L_{\cout}$–Lipschitz and $\mathcal{C}^1$ with $\beta_{\cout}$–Lipschitz Jacobian;
    \item $\sinn : \mathbb{R}^m \to \mathbb{R}^d$ is $L_{\sinn}$–Lipschitz and $\mathcal{C}^2$ with $\|\nabla^2\sinn_i(u)-\nabla^2\sinn_i(v)\|\le \gamma_{\sinn}\|u-v\|$;
    \item $\sout : \mathbb{R}^n \to \mathbb{R}$ is $\mathcal{C}^2$ with $\|\nabla^2\sout(u)-\nabla^2\sout(v)\|\le \gamma_{\sout}\|u-v\|$;
    \item $\Cin : \mathbb{R}^m \to \mathbb{R}^n$ is $\mathcal{C}^1$ with $\beta_{\Cin}$–Lipschitz Jacobian 
\end{itemize}
\end{assumption}

\begin{proposition}[Inner curvature cancellation]\label{prop:inner_block}
Let $J_k:=J_{\sinn}(x_k)$ and $y:=\nabla\cout(\sinn(x_k))$.
Define
\[
H_{\sinn,k}:=\sum_{i=1}^d y_i\,\nabla^2\sinn_i(x_k),
\]
and the model
\[
F^{\sinn}_{ H_{\sinn,k}}(x;x_k)\ :=\ \cout\big(\sinn(x_k)+J_k d\big)\;+\;\tfrac12\,d^\top H_{\sinn,k} d,
\qquad d:=x-x_k.
\]
Then, for $x$ near $x_k$,
\begin{align*}
|\cout(\sinn(x)) - F^{\sinn}_{ H_{\sinn,k} }(x;x_k)|
&\le M_{\sinn}^3\,\|d\|^3
\;+\;M_{\sinn}^4\,\|d\|^4,
\end{align*}
where $M_{\sinn}^3=\tfrac{L_{\cout}\gamma_{\sinn}}{6}+\tfrac{\beta_{\cout}L_{\sinn}\beta_{\sinn}}{2}$ and
$M_{\sinn}^4=\tfrac{\beta_{\cout}\beta_{\sinn}^2}{8} + \tfrac{\beta_{\cout}L_{\sinn}\gamma_{\sinn}}{6}$.
\end{proposition}

\begin{proof}
Let $q(d)\in\R^d$ with $[q(d)]_i=\langle \nabla^2\sinn_i(x_k)d,d\rangle$ and let $r_3$ be the cubic remainder in
\[
\sinn(x_k+d)=\sinn(x_k)+J_k d+\tfrac12\,q(d)+r_3,\qquad
\|q(d)\|\le \beta_{\sinn}\|d\|^2,\quad
\|r_3\|\le \tfrac{\gamma_{\sinn}}{6}\|d\|^3.
\]
Let $z:=\sinn(x_k)+J_k d$ and $\delta:=\tfrac12 q(d)+r_3$, so $\sinn(x)=z+\delta$.
By Taylor’s theorem with $\beta_{\cout}$–Lipschitz gradient,
\[
\cout(z+\delta)=\cout(z)+\langle \nabla\cout(z),\delta\rangle
+\tfrac12\,\delta^\top\big(\nabla^2\cout(z+\tau\delta)\big)\delta,
\quad \big\|\nabla^2\cout(z+\tau\delta)\big\|\le \beta_{\cout}.
\]
Add and subtract $\nabla\cout(\sinn(x_k))$:
\[
\cout(z+\delta)-\cout(z)
= \underbrace{\langle \nabla\cout(\sinn(x_k)),\delta\rangle}_{=\frac12\langle y,q(d)\rangle+\langle y,r_3\rangle}
+ \underbrace{\langle \nabla\cout(z)-\nabla\cout(\sinn(x_k)),\delta\rangle}_{\mathrm{I}}
+ \underbrace{\tfrac12\,\delta^\top\nabla^2\cout(\cdot)\,\delta}_{\mathrm{II}}.
\]
Since $\langle y,q(d)\rangle=d^\top H_{\sinn,k} d$,
\[
\cout(\sinn(x)) - F^{\sinn}_{ H_{\sinn,k} }(x;x_k)
= \langle y,r_3\rangle + \mathrm{I} + \mathrm{II}.
\]

Bounds:
$\|y\|=\|\nabla\cout(\sinn(x_k))\|\le L_{\cout}$ gives
$|\langle y,r_3\rangle|\le (L_{\cout}\gamma_{\sinn}/6)\|d\|^3$.
Moreover,
$\|\nabla\cout(z)-\nabla\cout(\sinn(x_k))\|\le \beta_{\cout}\|z-\sinn(x_k)\|
=\beta_{\cout}\|J_k d\| \le \beta_{\cout}L_{\sinn}\|d\|$ and
$\|\delta\|\le \tfrac12\|q(d)\|+\|r_3\|
\le \tfrac12\beta_{\sinn}\|d\|^2 + \tfrac{\gamma_{\sinn}}{6}\|d\|^3$, hence
\[
|\mathrm{I}|\le \tfrac{\beta_{\cout}L_{\sinn}\beta_{\sinn}}{2}\|d\|^3
+ \tfrac{\beta_{\cout}L_{\sinn}\gamma_{\sinn}}{6}\|d\|^4.
\]
Finally, $|\mathrm{II}|\le \tfrac12\beta_{\cout}\|\delta\|^2
\le \tfrac{\beta_{\cout}\beta_{\sinn}^2}{8}\|d\|^4+O(\|d\|^5)$.
Collecting the cubic and quartic terms and absorbing higher-order terms gives the stated bound.
\end{proof}

\begin{proposition}[Outer curvature cancellation]\label{prop:outer_block}
Fix $x_k$ and write $d:=x-x_k$, $\Delta:=\Cin(x)-\Cin(x_k)$, $J_k:=J_{\Cin}(x_k)$,
and $e_{\Cin}(x):=\Cin(x)-\Cin(x_k)-J_k d$.
Set
\[
\tilde H_{\sout,k}:=\nabla^2\sout(\Cin(x_k)),\quad
H_{\sout,k}:=J_k^\top \tilde H_{\sout,k} J_k,
\]
and define
\[
F^{\sout}_{H_{\sout,k}}(x;x_k)
:=\sout(\Cin(x_k))+\nabla\sout(\Cin(x_k))^\top\big(\Cin(x)-\Cin(x_k)\big)
+\tfrac12\, 
d^\top H_{\sout,k} d.
\]
Then, for $x$ near $x_k$,
\begin{align*}
| \sout(\Cin(x)) - F^{\sout}_{H_{\sout,k}}(x;x_k) |
&\le M_{\sout}^3\|d\|^3
\;+\;M_{\sout}^4\|d\|^4,
\end{align*}
where $M_{\sout}^3 = \tfrac{\beta_{\sout}L_{\Cin}\beta_{\Cin}}{2}+\tfrac{\gamma_{\sout}L_{\Cin}^3}{6}$ and
$M_{\sout}^4 = \tfrac{\beta_{\sout}\beta_{\Cin}^2}{8}$.
\end{proposition}

\begin{proof}
Second-order Taylor of $\sout$ at $\Cin(x_k)$ gives
\[
\sout(\Cin(x_k)+\Delta)
=\sout(\Cin(x_k))+\nabla\sout(\Cin(x_k))^\top\Delta
+\tfrac12\,\Delta^\top \tilde H_{\sout,k}\,\Delta + R_3,\quad
|R_3|\le \tfrac{\gamma_{\sout}}{6}\,\|\Delta\|^3.
\]
With $\Delta=J_k d+e$ and $e:=e_{\Cin}(x)$,
\[
\tfrac12\,\Delta^\top \tilde H_{\sout,k}\,\Delta
= \tfrac12\,d^\top H_{\sout,k} d + (J_k d)^\top \tilde H_{\sout,k} e + \tfrac12\,e^\top \tilde H_{\sout,k} e.
\]
Subtracting $F^{\sout}_{H_{\sout,k}}$ yields
\[
\sout(\Cin(x)) - F^{\sout}_{H_{\sout,k}}(x;x_k)
= (J_k d)^\top \tilde H_{\sout,k} e + \tfrac12\,e^\top \tilde H_{\sout,k} e + R_3.
\]
Using $\|\tilde H_{\sout,k}\|\le \beta_{\sout}$ gives
$\pm (J_k d)^\top \tilde H_{\sout,k} e \le \beta_{\sout}\|J_k d\|\,\|e\|$ and
$\pm \tfrac12 e^\top\tilde H_{\sout,k} e \le \tfrac12 \beta_{\sout}\|e\|^2$,
while $|R_3|\le (\gamma_{\sout}/6)\|\Delta\|^3$. Thus
\[
|\sout(\Cin(x)) - F^{\sout}_{H_{\sout,k}}(x;x_k)|
\ \le\ \beta_{\sout}\,\|J_k d\|\,\|e_{\Cin}(x)\|
\;+\;\tfrac12\,\beta_{\sout}\,\|e_{\Cin}(x)\|^2
\;+\;\tfrac{\gamma_{\sout}}{6}\,\|\Delta\|^3.
\]
Substituting the auxiliary estimates $\|J_k d\|\le L_{\Cin}\|d\|$, $\|e_{\Cin}(x)\|\le \tfrac{\beta_{\Cin}}{2}\|d\|^2$, and $\|\Delta\|\le L_{\Cin}\|d\|$ yields the claimed bound.
\end{proof}

\begin{corollary}[Projection after summation]\label{cor:joint-proj}
Adopt Assumption~\ref{asm:H_smth} and the notation of
Propositions~\ref{prop:inner_block} and~\ref{prop:outer_block}.
Let
\[
H_k\ :=\ H_{\sinn,k}+H_{\sout,k},\qquad
H_k^+\ :=\ \Pi_{\mathbb S_+}(H_k),\qquad
H_k^-\ :=\ H_k^+ - H_k\ \succeq 0 .
\]
Define the joint Hessian-augmented model
\begin{align*}
F_{H_k^+}(x;x_k)
\ :=\
\cvx(x) + 
\underbrace{\cout \big(\sinn(x_k)+J_{\sinn}(x_k)(x-x_k)\big)}_{\text{inner first order}}
\ &+ \
\underbrace{\sout(\Cin(x_k))
+\nabla\sout(\Cin(x_k))^\top\!\big(\Cin(x)-\Cin(x_k)\big)}_{\text{outer first order}} \\
\ &+\ \tfrac12\,(x-x_k)^\top H_k^+ (x-x_k).
\end{align*}
Then, for $x$ near $x_k$ and $d:=x-x_k$,
\begin{align*}
\text{\emph{(Upper bound)}}\quad
F(x) - F_{H_k^+}(x;x_k)
&\le \big(M_{\sout}^3+M_{\sinn}^3\big)\,\|d\|^3
\;+\;\big(M_{\sout}^4+M_{\sinn}^4\big)\,\|d\|^4,\\[2mm]
\text{\emph{(Lower bound)}}\quad
F(x) - F_{H_k^+}(x;x_k)
&\ge -\tfrac12\,d^\top H_k^- d
\;-\;\big(M_{\sout}^3+M_{\sinn}^3\big)\,\|d\|^3
\;-\;\big(M_{\sout}^4+M_{\sinn}^4\big)\,\|d\|^4,
\end{align*}
where the block constants $M_{\sout}^3,M_{\sout}^4$ and $M_{\sinn}^3,M_{\sinn}^4$ are those in
Propositions~\ref{prop:inner_block} and~\ref{prop:outer_block}.
Equivalently, using $d^\top H_k^- d \le \|H_k^-\|\,\|d\|^2$, the lower bound can be written with $\|H_k^-\|\,\|d\|^2/2$.
\end{corollary}

\begin{proof}
Apply the inner and outer curvature expansions of Propositions~\ref{prop:inner_block} and~\ref{prop:outer_block} at $x_k$ and add the resulting identities.
The exact second-order terms sum to $\tfrac12\,d^\top(H_{\sinn,k}+H_{\sout,k})d=\tfrac12\,d^\top H_k d$.
Subtracting the model’s quadratic $\tfrac12\,d^\top H_k^+ d$ leaves
\[
-\tfrac12\,d^\top(H_k^+ - H_k)\,d\ =\ -\tfrac12\,d^\top H_k^- d,
\]
which is the sole second-order ``projection gap''. The remaining terms are exactly the cubic and quartic remainders already bounded in the two propositions, and they add linearly, yielding the stated upper and lower bounds. Using $d^\top H_k^- d \le \|H_k^-\|\,\|d\|^2$ gives the norm form.
\end{proof}

\end{document}